\documentclass[reqno,twoside,12pt]{article}
\usepackage{amssymb,amsfonts,amsthm,amsmath}
\usepackage[pagebackref=false]{hyperref}
\usepackage{enumitem}

\usepackage[hmargin=1in,vmargin=1in]{geometry}

\usepackage{cite}

\usepackage{color}

\def\beq{\begin{equation}}
\def\eeq{\end{equation}}
\def\beqs{\begin{equation*}}
\def\eeqs{\end{equation*}}

\newcommand{\ddt}{\frac d{dt}}

\newcommand{\tnum}{\rm(\roman*)}
\newcommand{\rnum}{\rm(\alph*)}

\newtheorem{theorem}{Theorem}[section]
\newtheorem{lemma}[theorem]{Lemma}

\newtheorem{corollary}[theorem]{Corollary}
\newtheorem{definition}[theorem]{Definition}
\newtheorem{assumption}[theorem]{Assumption}

\theoremstyle{definition}
\newtheorem{remark}[theorem]{Remark}
\newtheorem{example}[theorem]{Example}
\newtheorem{scenario}[theorem]{Scenario}

\newtheorem*{xnotation}{Notation}

\definecolor{darkred}{rgb}{.70,.12,.20}

\definecolor{darkgreen}{rgb}{.20,.52,.14}

\newcommand{\varep}{\varepsilon}

\newcommand{\LL}{\mathcal L}
\newcommand{\iln}{L}
\newcommand{\R}{\mathbb{R}}
\newcommand{\N}{\mathbb{N}}
\newcommand{\Z}{\mathbb{Z}}

\newcommand{\classP}{\mathcal P}

\newcommand{\bigo}{\mathcal{O}}

\numberwithin{equation}{section}

\title{\textbf{Asymptotic expansions with exponential, power, and logarithmic functions for non-autonomous nonlinear differential equations
}}

\author{Dat Cao$^1$ and Luan Hoang$^2$}

\begin{document}
\date{\vspace{-4ex}}

\maketitle

\begin{center}
\textit{$^1$Department of Mathematics and Statistics\\
Minnesota State University, Mankato\\
Mankato, MN 56001,  U. S. A.}\\
Email address:  \texttt{dat.cao@mnsu.edu}\\

\medskip
\textit{$^2$Department of Mathematics and Statistics\\
Texas Tech University\\
1108 Memorial Circle, Lubbock, TX 79409--1042, U. S. A.}\\
Email address:  \texttt{luan.hoang@ttu.edu}\\
\bigskip
{\large \today}
\end{center}



\begin{abstract}
This paper develops further and systematically the asymptotic expansion theory that was initiated by Foias and Saut in \cite{FS87}.
We study the long-time dynamics of a large class of dissipative systems of nonlinear ordinary differential equations with time-decaying  forcing functions. The nonlinear term can be, but not restricted to, any smooth vector field which, together with its first derivative, vanishes at the origin. The forcing function can be approximated, as time tends to infinity, by a series of functions which are coherent combinations of exponential, power and iterated logarithmic functions. We prove that any decaying solution admits an asymptotic expansion, as time tends to infinity,  corresponding to the asymptotic structure of the forcing function. Moreover, these expansions can be generated by more than two base functions and go beyond the polynomial formulation imposed in previous work. 
\end{abstract}

\pagestyle{myheadings}\markboth{D. Cao and L. Hoang}
{Asymptotic expansions for nonlinear differential equations}

\tableofcontents

 \section{Introduction}

This work is motivated by a deep result by Foias and Saut \cite{FS87}, which is on the long-time behavior of solutions of the Navier--Stokes equations,  and its later developments in \cite{FS91,HM1,HM2,CaH1,CaH2,Minea,Shi2000,HTi1}. In the original work \cite{FS87}, Foias and Saut studied the Navier--Stokes equations written in the functional form (on an appropriate infinite dimensional space) as 
\beq\label{NSE}
u_t+Au+B(u,u)=0,
\eeq
where $A$ is a linear operator, and $B$ is a bi-linear form.
They established the following asymptotic expansion, as $t\to\infty$,
\beq\label{nsexp}
u(t)\sim \sum_{k=1}^\infty q_k(t)e^{-\mu_k t},
\eeq
where $q_k(t)$'s are polynomials in $t$, and $\mu_k$ increases to infinity. Roughly speaking, expansion \eqref{nsexp} means that, for each $N$, the solution $u(t)$ can be approximated by the finite sum 
\beqs
s_N(t):=\sum_{k=1}^N q_k(t)e^{-\mu_k t},
\eeqs 
in the sense that the remainder $u(t)-s_N(t)$ decays exponentially faster than the fastest decaying mode $e^{-\mu_N t}$ in $s_N(t)$, see Definition \ref{psi-phi} below for the precise meaning.

Expansion \eqref{nsexp} is studied in more details in \cite{FS91,FHN1,FHN2,FHOZ1,FHOZ2,FHS1,HM1} regarding its convergence, approximation in Gevrey spaces, associated invariant nonlinear manifolds and normal form, and connection to the theory of Poicar\'e--Dulac normal form, applications to statistical solutions and turbulence theory, etc. A similar expansion to \eqref{nsexp} is also established in \cite{HTi1} for the Navier--Stokes equations of rotating fluids. Besides the Navier--Stokes equations, expansion \eqref{nsexp} were obtained and studied for other ordinary differential equations (ODEs) \cite{Minea}, and dissipative wave equations \cite{Shi2000}. The last two mentioned papers deal with equations with more general nonlinearity than the quadratic term  $B(u,u)$ in \eqref{NSE}. However, they are still autonomous systems. 

Regarding non-autonomous systems, recent papers \cite{HM2,CaH1,CaH2} extend the Foias--Saut result to the Navier--Stokes equations  with time-dependent forces, that is, 
\beq\label{Nf}
u_t+Au+B(u,u)=f(t),
\eeq
where the force $f(t)$ decays to zero as $t\to\infty$.
In \cite{HM2}, asymptotic expansion \eqref{nsexp} for a solution $u(t)$ of \eqref{Nf} is obtained under the condition that
\beq\label{hmf}
f(t)\sim \sum_{k=1}^\infty p_k(t)e^{-\mu_k t},
\eeq
where $p_k$'s are appropriate polynomials.
The papers \cite{CaH1,CaH2} consider the forces that decay not as fast as exponential functions. It is obtained, among other things, that if 
\beq\label{cahf}
f(t)\sim \sum_{k=1}^\infty \eta_k t^{-\mu_k}, \text{ respectively, } 
f(t)\sim \sum_{k=1}^\infty \eta_k (\ln t)^{-\mu_k},
\eeq 
where $\eta_k$'s are constant vectors (in functional spaces), then there exist constant vectors $\xi_k$'s  such that 
\beq\label{cahu}
u(t)\sim \sum_{k=1}^\infty \xi_k t^{-\mu_k}, \text{ respectively, } 
u(t)\sim \sum_{k=1}^\infty \xi_k (\ln t)^{-\mu_k}.
\eeq 

However, the fact that $\eta_k$ and $\xi_k$ are independent of $t$ makes the expansions in \eqref{cahf} and \eqref{cahu} less than full counterparts of the original \eqref{nsexp}.

The current paper aims to combine two approaches: one in \cite{Minea,Shi2000} for general equations, and one in \cite{HM2,CaH1,CaH2} for general forcing functions. To make the ideas clear, we avoid, in this paper, complicated issues about global existence, uniqueness, and regularity that often arise in nonlinear partial differential equations (PDEs). Thus, we choose to work with systems of ODEs (in finite dimensional spaces) with general nonlinearity, and explore various types of forcing functions.
We describe the systems of differential equations of our interest and explain the main ideas now.

\begin{xnotation}
The following notation will be used throughout the paper. 

\begin{itemize}
 \item $\N=\{1,2,3,\ldots\}$ denotes the set of natural numbers, and $\Z_+=\N\cup\{0\}$.

 \item For a vector $y\in\R^n$, its Euclidean norm is denoted by $|y|$.
 
 \item Let $f$ be a non-negative function defined in a neighborhood of the origin in $\R^n$. For a number $\alpha>0$, we write $f(y)=\bigo(|y|^\alpha)$ as $y\to 0$, if there are positive numbers  $r$ and $C$ such that $f(y)\le C|y|^\alpha$ for all $x\in\R^n$ with $|x|<r$.   
 
 \item Let $f,h:[T_0,\infty)\to[0,\infty)$ for some $T_0\in \R$. We write $f(t)=\mathcal O(h(t))$ (implicitly means as $t\to\infty$) if there exist numbers $T\ge T_0$ and $C>0$ such that $f(t)\le Ch(t)$ for all $t\ge T$.

 \item Let $T_0\in\R$, functions $f,g:[T_0,\infty)\to\R^n$, and $h:[T_0,\infty)\to[0,\infty)$. 
We will conveniently write  $f(t)=g(t)+\bigo(h(t))$ to indicate
  $|f(t)-g(t)|=\bigo(h(t))$.  
  
  In particular, when $g=0$, the expression $f(t)=\bigo(h(t))$ means $|f(t)|=\bigo(h(t))$. 

\item For long-time estimates, we will algebraically manipulate the above big-O notation in our calculations. For example, suppose $u(t)$ and $v(t)$ are $\R^n$-valued functions with $u(t)=\bigo(e^{-t})$ and $v(t)=\bigo(e^{-2t})$. We can manipulate (and read from left to right)
$$u(t)+v(t)=\bigo(e^{-t})+\bigo(e^{-2t})=\bigo(e^{-t}),\quad |u(t)|v(t)=\bigo(e^{-t})\bigo(e^{-2t})=\bigo(e^{-3t}).$$
 \end{itemize}
\end{xnotation}

Let $n\in \N$ be fixed throughout the paper.
Consider the following system of nonlinear ODEs in $\R^n$:
\beq \label{sys-eq}
y'=-Ay +G(y)+f(t),
\eeq 
where $A$ is an $n\times n$ constant (real) matrix, $G$ is a vector field on $\R^n$, and  $f$ is a function from $(0,\infty)$ to $\R^n$.
 
\begin{assumption}\label{assumpA}
Matrix $A$ is a diagonalizable matrix with positive eigenvalues.
 \end{assumption}
 
This assumption is common in studying the dissipative dynamical systems, although it is not as general as \cite{Shi2000}. It helps us simplify the calculations and displays the key features of the dissipative dynamics.
 
\begin{assumption}\label{assumpG}  Function $G:\R^n\to\R^n$ has the the following properties. 
\begin{enumerate}[label=\tnum]
 \item  $G$ is locally Lipschitz.
 \item There exist functions $G_m:\R^n\to\R^n$,  for $m\ge 2$, each is a homogeneous polynomial of degree $m$, such that, for any $N\ge 2$, there exists $\delta >0$ such that
 \beq\label{Grem1}
 \Big|G(y)-\sum_{m=2}^N G_m(y)\Big|=\bigo(|y|^{N+\delta})\text{ as } y\to 0.
 \eeq
 \end{enumerate}
\end{assumption}

We formally write Assumption \ref{assumpG}(ii) as an expansion
\beq \label{Gex}
G(y)\sim \sum_{m=2}^\infty G_m(y)\text{ as }y\to 0.
\eeq 

It is clear that if $G$ is a $C^\infty$-function with $G(0)=0$ and $G'(0)=0$ then $G$ satisfies Assumption \ref{assumpG}.  
Note that we do not require the convergence of the formal series on the right-hand side of \eqref{Gex}.
Even when the convergence occurs, the limit is not necessarily the function $G(y)$. For instance, if $h:\R^n\to\R^n$ satisfies $|y|^{-\alpha} h(y)\to 0$ as $y\to 0$ for all $\alpha>0$, then $G$ and $G+h$ have the same expansion \eqref{Gex}.

Next, we investigate the class of functions for $f(t)$ and the forms of expansions that can be obtained.
Since this paper involves different vector-valued polynomials of several variables, we clarify their definition here.

\begin{definition}\label{polydef}
 Let $X$ be a real linear space.  For $m\in\N$, a function $p:\R^m\to X$ is a polynomial if
 \beqs
 p(z)=\sum a_\alpha z^\alpha \text{ for }z\in\R^m,
 \eeqs
 where the sum is taken over finitely many multi-index $\alpha\in \Z_+^m$, and $a_\alpha$'s are vectors in $X$.

 In particular, when $m=1$, a function $p:\R\to X$ is a polynomial if
 \beqs
 p(t)=\sum_{k=0}^N a_k t^k \text{ for }t\in\R,
 \eeqs
 where $N\ge 0$, and $a_k$'s are vectors in $X$.
 \end{definition}

 Examining the expansions in \eqref{nsexp}, \eqref{hmf}, \eqref{cahf} and \eqref{cahu}, we aim to establish some results for the forcing function of the general form 
  \beq\label{fform}
f(t) \sim \sum_{k=1}^\infty p_k(\phi(t)) \psi(t)^{-\gamma_k},
 \eeq 
 where $p_k$'s are $\R^n$-valued polynomials of one variable, and $\gamma_k$'s are positive numbers.

Clearly, when $\psi(t)=e^t$ and $\phi(t)=t$,  \eqref{fform} gives \eqref{hmf}.
When $\psi(t)=t$, resp., $\psi(t)=\ln t$,  \eqref{fform} resembles \eqref{cahf}. 
Because of the presence of the time-dependent function $\phi(t)$ in \eqref{fform}, it will remove the restriction of having only constant vectors $\eta_k$'s in \eqref{cahf}.
However, it is not clear what $\phi(t)$ might be. 
Nonetheless, we provide the rigorous definition for \eqref{fform} here and will specify some natural choices for $\phi(t)$ later.

\begin{definition}\label{psi-phi}
Let $(\psi,\phi)$ be a pair of real-valued functions defined on $(T,\infty)$ for some $T\in \R$ such that 
\beq\label{phs1}
\lim_{t\to\infty}\psi(t) =\infty =\lim_{t\to\infty}\phi(t) ,\eeq  
and
 \beq\label{compcond}
 \lim_{t\to\infty} \frac{\phi^\lambda(t)}{\psi(t)}=0\text{ for all }\lambda>0.
 \eeq
 
 Let $(X,\|\cdot\|_X)$ be a normed space, and $g$ be a function from $(T',\infty)$ to $X$ for some $T'\in\R$.

\begin{enumerate}[label=\tnum] 
 \item Let $(\gamma_k)_{k=1}^\infty$ be a divergent, strictly increasing sequence of positive numbers.
We say
  \beq \label{expan}
  g(t) \sim \sum_{k=1}^\infty p_k(\phi(t)) \psi(t)^{-\gamma_k},
  \eeq
  where each $p_k:\R \to X$ is a polynomial,   if for any $N\ge 1$, there exists $\mu>\gamma_N$ such that
  \beq\label{gdrem}
  \Big\|g(t)-\sum_{k=1}^N p_k(\phi(t)) \psi(t)^{-\gamma_k}\Big\|_X=\bigo(\psi(t)^{-\mu}).  
  \eeq
 
 \item Let $N\in\N$, and $(\gamma_k)_{k=1}^N$ be positive and strictly increasing.
We say
  \beq\label{finex}
g(t) \sim \sum_{k=1}^N p_k(\phi(t)) \psi(t)^{-\gamma_k},
\eeq
where each $p_k:\R \to X$ is a polynomial, for $1\le k\le N$, if
\beq\label{allam}
\Big \|g(t) - \sum_{k=1}^N p_k(\phi(t)) \psi(t)^{-\gamma_k}\Big \|_X=\bigo(\psi(t)^{-\lambda})\quad \text{ for any }\lambda>0.
\eeq
\end{enumerate}

We call $\psi$ and $\phi$ the \emph{base functions} for expansions \eqref{expan} and \eqref{finex}, with $\psi$ being \emph{primary}, and $\phi$ being \emph{secondary}. 
\end{definition}
  
In case $X$ is a finite dimensional normed space, all norms on $X$ are equivalent. Hence, the above definitions of \eqref{expan} and \eqref{finex} are independent of the particular norm $\|\cdot\|_X$. 

One can see, again, that expansions \eqref{nsexp} and \eqref{hmf} correspond to \eqref{expan} with $(\psi,\phi)=(e^t,t)$.
The next two examples are $(\psi,\phi)=(t,\ln t)$ and $(\psi,\phi)=(\ln t,\ln \ln t)$, which cover the expansions in \eqref{cahf} and \eqref{cahu}. 
It turns out that, corresponding to these three cases of $(\psi,\phi)$ and asymptotic expansion \eqref{fform} for $f(t)$,  we can prove that 
 any decaying solution $y(t)$ of \eqref{sys-eq} admits a similar asymptotic expansion 
 \beq\label{yform}
 y(t) \sim \sum_{k=1}^\infty q_k(\phi(t)) \psi(t)^{-\mu_k},
 \eeq 
 where $\mu_k$'s are positive numbers appropriately generated based on the powers $\gamma_k$'s in \eqref{fform}.

 This is the starting point of the current paper that explains the main ideas. More sophisticated expansions and all technicalities will be presented in details below.
 
The paper is organized as follows.
Section \ref{ode-lemmas} sets up the background for equation \eqref{sys-eq} and develops essential tools for the paper. It contains the approximation lemmas for solutions of linear ODEs, see Lemmas \ref{exp-odelem}, \ref{iterlog-odelem} and Corollary \ref{powlog-odelem}. They are simple but important building blocks for the complicated nonlinear theory developed in later sections.
Especially,  Lemma  \ref{iterlog-odelem} will enable us to deal with a much larger class of forcing functions.
In  Section \ref{odesec}, we prove the basic existence result in Theorem \ref{thmsmall} for solutions of the studied ODEs.
A specific asymptotic estimate, as time tends to infinity, which corresponds to the decay of the forcing function, is established in Theorem \ref{thmdecay}. These will be used to obtain the first term in the asymptotic expansion of the solutions. 
Section \ref{basetwo} contains our first main result, Theorem \ref{mainthm}, corresponding to the expansions  in Definition \ref{psi-phi}, with specific types of $(\psi,\phi)$ indicated in Definition \ref{typedef}. It fully justifies \eqref{yform} and removes the limitation of the previous work \cite{CaH1,CaH2} mentioned in the remark after \eqref{cahu}.
Moreover, we emphasize that the calculations in the proof of Theorem \ref{mainthm} will crucially serve the further developments in the next section.
In Section \ref{multisec}, we investigate expansions that are generated by more than two base functions, see Definition \ref{mixdef}. 
They can consist of many secondary base functions and allow the functions $p_k$'s in \eqref{expan} to be more than just polynomials, i.e., the powers can be real numbers, not just non-negative integers. Therefore, compared to those in Section \ref{basetwo} with the same primary base function, these expansions are more precise approximations of the forcing function and solutions.
 Despite not yet covering the case of exponential  primary base function, they are rather significant deviations from the polynomial-based formulation for $q_k$'s and $p_k$'s in the asymptotic expansions \eqref{nsexp} and \eqref{expan}, respectively. 
The case of purely iterated logarithmic functions is treated in subsection \ref{logsec}, see Theorem \ref{logthm}.
The case of mixed power and iterated logarithmic functions is treated in subsection \ref{doublexp}, see Theorem \ref{mixpl}.
Typical cases of expansions with the triple power-log-loglog functions, see \eqref{uqln} and \eqref{uMul}, are explored in Corollaries \ref{cor54} and \ref{cor55}. In Appendix \ref{append}, we briefly discuss a specific aspect of our results, namely, their application to the Galerkin approximations to nonlinear PDEs.
  
 We comment that the approach presented in the current paper is based on the result and ideas of Foias and Saut in \cite{FS87}. 
 For our specific problem of obtaining  large-time  asymptotic expansions for decaying solutions, it is direct, simple and does not resort to the normal form theory for ODEs, see, e.g., \cite{ArnoldODEGeo,Bruno1,MurdockBook}. 

Finally, it is worth mentioning that our results can be extended to the Navier--Stokes equations or other PDEs with appropriate settings, by combining this paper's techniques with the methods in \cite{CaH1,CaH2} or \cite{Shi2000}.
 
\section{Preliminaries} \label{ode-lemmas}

If $m\in \N$ and $\mathcal L$ is an $m$-linear mapping from $\R^{m\times n}$ to $\R^n$, the norm of $\mathcal L$ is defined by
\beqs
\|\mathcal L\|=\max\{ |\mathcal L(y_1,y_2,\ldots,y_m)|:y_j\in\R^n,|y_j|=1,\text{ for } 1\le j\le m\}.
\eeqs
Then $\|\mathcal L\|$ is a number in $[0,\infty)$, and, for any $y_1,y_2,\ldots,y_m\in\R^n$, one has 
\beq\label{multiL}
|\mathcal L(y_1,y_2,\ldots,y_m)|\le \|\mathcal L\|\cdot |y_1|\cdot |y_2|\cdots |y_m|.
\eeq

First, we examine equation \eqref{sys-eq} further.
Regarding its linear part, thanks to Assumption \ref{assumpA}, we can denote by $\Lambda_k$, for $1\le k\le n$, the eigenvalues of $A$ which are positive and increasing in $k$.
Then there exists an invertible matrix $S$ such that
\beq \label{Adiag}
A=S^{-1} A_0 S,
\text{ where } A_0={\rm diag}[\Lambda_1,\Lambda_2,\ldots,\Lambda_n].
\eeq 

Now, denote the distinct eigenvalues of $A$ by $\lambda_j$'s which are strictly increasing in $j$, i.e., 
 \beqs
 0<\lambda_1=\Lambda_1<\lambda_2<\ldots<\lambda_d=\Lambda_n,\quad \text{ with } 1\le d\le n.
 \eeqs

For $1\le k,\ell\le n$, let $E_{k\ell}$ be the elementary $n\times n$ matrix $(\delta_{ki}\delta_{\ell j})_{1\le i,j\le n}$, where $\delta_{ki}$ and $\delta_{\ell j}$ are the Kronecker delta symbols.

For an eigenvalue $\lambda$ of $A$, define
 \beqs
\hat R_\lambda=\sum_{1\le i\le n,\Lambda_i=\lambda}E_{ii}\text{ and } R_\lambda=S^{-1}\hat R_\lambda S.
 \eeqs
 
The following are immediate facts.
\begin{enumerate}[label=\rnum]
\item If $\lambda$ is an eigenvalue of $A$, then 
\beq\label{Pa} R_\lambda^2=R_\lambda,
\eeq 
\beq \label{Pd}   AR_\lambda=R_\lambda A=\lambda R_\lambda.
\eeq 

\item If $\lambda$ and $\mu$ are distinct eigenvalues of $A$, then
\beq \label{Pb}  R_\lambda R_\mu=0.
\eeq 

\item One has
\beq \label{Pc} I_n = \sum_{j=1}^d R_{\lambda_j},
\eeq 
and, for any $y\in\R^n$, 
\beq\label{Pe} 
|y|\le \sum_{j=1}^d |R_{\lambda_j}y|\le (\sum_{j=1}^d \|R_{\lambda_j}\| ) |y|.
\eeq
\end{enumerate}

For the nonlinear part of equation \eqref{sys-eq}, we consider condition \eqref{Gex}. For each $m\ge 2$,  there exists an $m$-linear mapping $\mathcal G_m$ from $\R^{m\times n}$ to $\R^n$ such that
\beq\label{GG} G_m(y)=\mathcal G_m(y,y,\ldots,y)\text{ for } y\in\R^n.\eeq

By \eqref{multiL}, one has, for any $y_1,y_2,\ldots,y_m\in\R^n$, that
\beq\label{multineq}
|\mathcal G_m(y_1,y_2,\ldots,y_m)|\le \|\mathcal G_m\|\cdot |y_1|\cdot |y_2|\cdots |y_m|.
\eeq

In particular,
\beq\label{multy}
|G_m(y)|\le \|\mathcal G_m\|\cdot |y|^m\quad\forall y\in\R^n.
\eeq

It follows \eqref{Grem1}, when $N=2$, and \eqref{multy}, for $m=2$, that $|G(y)|=\bigo(|y|^2)$ as $y\to 0$. Thus, there exists numbers $c_*,r_*>0$ such that 
\beq \label{Gyy}
|G(y)|\le c_*|y|^2 \quad\forall y\in\R^n \text{ with } |y|<r_*.
\eeq

Next, we obtain elementary results on long-time asymptotic estimates for integrals and approximations for solutions of linear ODEs.
They will play important roles in later developments of the paper.

We will often use the following simple fact
\beq\label{ee}
\int_0^t e^{-\alpha(t-\tau)}e^{-\beta\tau}d\tau 
=\begin{cases}
 \bigo(e^{-\min\{\alpha,\beta\}t}),    &\text{ if } \alpha\ne \beta,\\
  \bigo(t e^{-\alpha t}),  &\text{ if } \alpha=\beta.
  \end{cases}
\eeq

We recall here and make concise  \cite[Lemma 4.2]{HM2}, which originates from the work of Foias and Saut \cite{FS87}.

\begin{lemma}\label{polylem}
Let $(X,\|\cdot\|_X)$ be a Banach space.
Let $p:\R\to X$ be a polynomial, and $g\in C([0,\infty),X)$ satisfy 
 \beqs
 \|g(t)\|_X\le Me^{-\delta t} \quad \forall t\ge 0, \quad \text{for some } M,\delta>0.
 \eeqs

Suppose  that $y\in C([0,\infty),X)$ solves the equation
 \beqs
 y'(t)+ \beta y(t) =p(t)+g(t), \quad  \text{ for }t > 0,
 \eeqs
 where $\beta$ is a constant in $\R$.

 In case $\beta<0$, assume further that
 \beq\label{yexpdec}
  \lim_{t\to\infty} (e^{\beta t}\|y(t)\|_X)=0.
  \eeq
  
 Then there exists a unique $X$-valued polynomial $q(t)$ such that 
\beq\label{pode2}
q'(t)+\beta q(t) = p(t),\quad \text{for }t\in \R,
\eeq
and 
 \beq\label{gball}
  \|y(t)-q(t)\|_X=\bigo(e^{-\varep t}) \quad \text{ for any }\varep\in(0,\varep_{\delta,\beta}),
  \eeq
where
\beq\label{epp}
\varep_{\delta,\beta}=\begin{cases}
                       \min\{\delta,\beta\}&\text{ if } \beta>0,\\
                       \delta&\text{ otherwise.}
                      \end{cases}
\eeq
\end{lemma}
\begin{proof}
The uniqueness of $q(t)$ is due to Lemma \ref{ue} below with $N=1$, $\gamma_1=0$, $(\psi,\phi)=(e^t,t)$ and estimate \eqref{gball}.
In fact, following \cite[Lemma 4.2]{HM2}, the polynomial $q(t)$ is explicitly defined by 
\beq\label{qdef}        
 q(t)=
\begin{cases}
e^{-\beta t}\int_{0}^t e^{\beta\tau }p(\tau) d\tau&\text{if }\beta >0,\\
y(0) +\int_0^\infty  g(\tau)d\tau + \int_0^t p(\tau)d\tau &\text{if }\beta =0,\\
-e^{-\beta t}\int_t^\infty e^{\beta\tau }p(\tau) d\tau&\text{if }\beta <0, 
\end{cases}
\eeq
and satisfies the following estimates.

If $\beta>0$ then
  \beq\label{g1b1}
  \|y(t)-q(t)\|_X\le \Big(\|y(0)-q(0)\|_X + \frac{M}{|\beta-\delta|}\Big)e^{-\min\{\delta,\beta\} t}, \quad t\ge 0,\text{ for } \beta\ne \delta,
  \eeq
  and
  \beq\label{g1b1equal}
  \|y(t)-q(t)\|_X\le (\|y(0)-q(0)\|_X + M t )e^{-\delta t}, \quad t\ge 0,\text{ for } \beta=\delta.
  \eeq

If either ($\beta=0$)  or ($\beta<0$ with \eqref{yexpdec}),  then
  \beq\label{g1b3}
  \|y(t)-q(t)\|_X\le \frac{M}{|\beta|+\delta}e^{-\delta t}, \quad t\ge 0.
  \eeq

Combining \eqref{g1b1}, \eqref{g1b1equal} and \eqref{g1b3}, we deduce the concise estimate \eqref{gball}. 
\end{proof}

Using the basic Lemma \ref{polylem}, we derive, below, an efficient approximation lemma for linear ODEs. It will be utilized in the proof of the main result of Section \ref{basetwo}, Theorem \ref{mainthm}.

\begin{lemma}\label{exp-odelem}
Let $p(t)$ be an $\R^n$-valued polynomial and  $g:[T,\infty)\to \R^n$, for some  $T\in\R$, be a continuous function satisfying $|g(t)|=\bigo(e^{-\alpha t})$ for some $\alpha>0$.
Suppose $\lambda > 0$ and $y\in C([T,\infty),\R^n)$ is a solution of 
 \beq\label{odey}
 y'(t)=-(A-\lambda I_n)y(t)+p(t)+g(t),\quad \text{for }t\in(T,\infty).
 \eeq

If $\lambda>\lambda_1$, assume further that
\beq\label{ellams}
\lim_{t\to\infty} (e^{(\lambda_* -\lambda )t}|y(t)|)=0,\text{ where } \lambda_*=\max\{\lambda_j: 1\le j\le d,\lambda_j<\lambda\}.
\eeq 
 
Then there exists a unique $\R^n$-valued polynomial $q (t)$ such that 
\beq\label{qeq} q'(t)=-(A-\lambda I_n) q (t) +  p (t),\quad  \text{for }t\in\R,
\eeq  
and
 \beq\label{yqest}
 |y(t)-q(t)|=\bigo(e^{-\varep  t}) \quad \text{ for all }\varepsilon \in(0,\delta),
 \eeq
 where
$ \delta=\min\{\varep_{\alpha,\lambda_j-\lambda}:1\le j\le d\}>0$, with $\varep_{\alpha,\lambda_j-\lambda}$ being defined as in \eqref{epp}.
 \end{lemma}
\begin{proof}
Similar to Lemma \ref{polylem}, the uniqueness of $q(t)$ comes from Lemma \ref{ue} below with $X=\R^n$, $N=1$, $\gamma_1=0$, $(\psi,\phi)=(e^t,t)$ and the exponential decay in \eqref{yqest}.

Let $1\le j\le d$. Applying $R_{\lambda_j}$ to equation \eqref{odey} and using property \eqref{Pd} give
\beq\label{Req1}
 (R_{\lambda_j}y(t))'=-(\lambda_j -\lambda)R_{\lambda_j} y(t) +R_{\lambda_j} p(t) + R_{\lambda_j} g(t),\quad\text{for } t>T.
\eeq
 
We apply Lemma \ref{polylem} to equation \eqref{Req1}  with $X=R_{\lambda_j}(\R^n)$ , $\|\cdot\|_X=|\cdot|$,
\beqs 
\beta:=\lambda_j-\lambda,\quad y(t):=R_{\lambda_j}y(T+t),\quad p(t):=R_{\lambda_j}p(T+t),\quad g(t):=R_{\lambda_j}g(T+t).
\eeqs 

In case $\beta<0$, we have 
$e^{\beta t}|R_{\lambda_j}y(T+t)|\le e^{(\lambda_* -\lambda )t}|y(T+t)|$, which goes to zero as $t\to\infty$, thanks to \eqref{ellams}. Thus, condition \eqref{yexpdec} is satisfied. Therefore, it follows \eqref{gball} and \eqref{pode2} that 
  \beq\label{Rq1}
  |R_{\lambda_j} y(T+t)-\tilde q_j(t)|= \bigo(e^{-\varep t})\quad \forall \varepsilon\in(0,\varep_{\alpha,\lambda_j-\lambda}),
  \eeq
where $\tilde q_j(t)$ is an $R_{\lambda_j}(\R^n)$-valued polynomial that satisfies 
\beq\label{Rq2}
\tilde q_j'(t)=- (\lambda_j- \lambda)\tilde q_j(t) +  R_{\lambda_j} p(T+t)
\quad \text{for }t\in\R.
\eeq

Define $q(t) =\sum_{j=1}^d \tilde q_j(t-T)$ for $t\in\R$. Then $q(t)$ is an $\R^n$-valued polynomial and, by properties \eqref{Pa} and \eqref{Pb}, $R_{\lambda_j}q(T+t)=\tilde q_j(t)$.
By using the last relation and replacing $T+t$ with $t$, we have from \eqref{Rq1} and \eqref{Rq2} that
 \beq \label{qj-est}
  |R_{\lambda_j}( y(t)-q(t))|= \bigo(e^{-\varep t}) \quad \forall \varep\in(0,\delta),
  \eeq
\beq\label{qjeq}
R_{\lambda_j} q'(t)=-(A- \lambda I_n)  R_{\lambda_j}q(t) +  R_{\lambda_j} p(t)
\quad \text{for }t\in\R.
\eeq

Summing up \eqref{qj-est} in $j$ from $1$ to $d$, and using the first inequality in \eqref{Pe}, we obtain \eqref{yqest}.

Finally, summing up \eqref{qjeq} in $j$ from $1$ to $d$, and using identity \eqref{Pc}, we obtain \eqref{qeq}.
\end{proof}

\begin{remark}\label{rmkoq}
We recall, below, some well-known observations from \cite{FS87,FS91}.
\begin{enumerate}[label=\rnum]
 \item\label{fsa}  Regarding Lemma \ref{polylem}, if $\beta\ne 0$, then the polynomial $q(t)$ is independent of solution $y(t)$, see formula \eqref{qdef}. In fact, it is the unique polynomial solution of \eqref{pode2}. 
In case $\beta=0$, $q(t)$ depends on $y(0)$.
 \item\label{fsb} Consequently, if $\lambda$ is not an eigenvalue of $A$, then the polynomial $q(t)$ in Lemma \ref{exp-odelem} is the unique polynomial solution of \eqref{qeq}, and independent of $y(t)$. In case $\lambda$ is an  eigenvalue of $A$, then $q(t)$ depends on $R_\lambda y(T)$.
 \end{enumerate}
\end{remark}

The following types of functions will be used in our asymptotic expansions.

\begin{definition}\label{ELdef} Define the iterated exponential and logarithmic functions as follows:
\begin{align*} 
&E_0(t)=t \text{ for } t\in\R,\text{ and } E_{m+1}(t)=e^{E_m(t)}  \text{ for } m\in \Z_+, \ t\in \R,\\
&\iln_0(t)=t\text{ for } t\in\R,\text{ and } \iln_{m+1}(t)= \ln(\iln_m(t)) \text{ for } m\in \Z_+,\ t>E_m(0).
\end{align*}

Let $m\in \Z_+$ and $k\in \N$, define $\R^k$-valued function $\LL_{m,k}$ by
\beqs 
\LL_{m,k}(t)=(\iln_{m+1}(t),\iln_{m+2}(t),\ldots,\iln_{m+k}(t)) \text{ for } t>E_{m+k-1}(0).
\eeqs

In particular, denote $\LL_k=\LL_{0,k}$ for $k\in \N$.
\end{definition}

Throughout the paper, we will use the convention 
\beqs
\sum_{k=1}^0 a_k=0\quad\text{and}\quad \prod_{k=1}^0 a_k=1.
\eeqs

For $m\in\Z_+$, note that $L_m(t)$ is positive and increasing  for $t>E_m(0)$, and
\beqs
\lim_{t\to\infty} \iln_m(t)=\infty,\quad \lim_{t\to\infty} \frac{\iln_{m+k}(t)^\lambda}{\iln_m(t)}=0\text{ for all }k\in\N, \ \lambda>0. 
\eeqs

For $m\in \N$, the first and second derivatives of $\iln_m(t)$ are
\begin{align*}
 L_m'(t)&=\frac 1{t\prod_{k=1}^{m-1} L_k (t)}, \\
L_m''(t)&=-\frac1{t^2 } \cdot \frac 1{\prod_{k=1}^{m-1} L_k (t)} \left\{1+\sum_{j=1}^{m-1}\frac 1{\prod_{p=1}^{j} L_p (t)} \right\}.
\end{align*}

Clearly, $\iln_m''(t)< 0$ on  $(E_m(0),\infty)$, hence,  
\beq \label{Lmcon}
L_m\text{ is concave on }(E_m(0),\infty) \text{ for all $m\in \N$.}
\eeq 

In the next two lemmas, we obtain asymptotic estimates for some integrals which will appear in later proofs. 

\begin{lemma}\label{plnlem}
Let $m\in\Z_+$ and $\lambda>0$, $\sigma>0$  be given. For any number $T_*>E_m(0)$, there exists a number $C>0$ such that
\beq\label{iine2}
 \int_0^t e^{-\sigma (t-\tau)}\iln_m(T_*+\tau)^{-\lambda}d\tau
 \le C \iln_m(T_*+t)^{-\lambda} \quad\text{for all }t\ge 0.
\eeq
\end{lemma}
\begin{proof}
We recall \cite[Lemma 3.1]{CaH1}, which states as follows.

\textit{ Let $F$ be a continuous, decreasing function from $[0,\infty)$ to $[0,\infty)$.
For any $\sigma>0$ and $\theta\in(0,1)$, one has  
\beq \label{Fi1}
 \int_0^t e^{-\sigma (t-\tau)}F(\tau)d\tau
 \le \frac1\sigma\Big(F(0)e^{-(1-\theta)\sigma t}+F(\theta t)\Big)\quad\forall t\ge 0.
\eeq
}

\emph{Case $m=0$.} Then $T_*>0$. Applying inequality \eqref{Fi1} to $F(t)=(T_*+t)^{-\lambda}$ and $\theta=1/2$ gives 
\beq \label{iine3}
 \int_0^t e^{-\sigma (t-\tau)}(T_*+\tau)^{-\lambda}d\tau
 \le \frac1\sigma\Big(T_*^{-\lambda}e^{-\sigma t/2}+(T_*+t/2)^{-\lambda}\Big)\quad\forall t\ge 0.
\eeq

Clearly, there exists a number $C_1>0$ such that 
 \beq\label{eF1}
 e^{-\sigma t/2}\le C_1 (T_*+t)^{-\lambda}\quad\text{for all } t\ge 0.
 \eeq

 Also, for $t\ge 0$,
 \beq\label{iine4}
 (T_*+t/2)^{-\lambda} = 2^\lambda(2T_*+t)^{-\lambda}\le  2^\lambda(T_*+t)^{-\lambda}.
 \eeq

From \eqref{iine3}, \eqref{eF1} and \eqref{iine4}, we obtain \eqref{iine2}. 

\medskip
\emph{Case $m\ge 1$.}  Let $F(t)=\iln_m(T_*+t)^{-\lambda}$ and fix $\theta\in(0,1)$.  By \eqref{Fi1},
\beq \label{iine5}
 \int_0^t e^{-\sigma (t-\tau)}\iln_m(T_*+\tau)^{-\lambda}d\tau
 \le \frac1\sigma\Big(\iln_m(T_*)^{-\lambda}e^{-(1-\theta)\sigma t}+\iln_m(T_*+\theta t)^{-\lambda}\Big)\quad\forall t\ge 0.
\eeq

Similar to \eqref{eF1}, there is a number $C_2>0$ such that
 \beq\label{eF2}
 e^{-(1-\theta)\sigma t}\le C_2 \iln_m(T_*+t)^{-\lambda}\quad\text{for all } t\ge 0.
 \eeq

By the concavity of $\iln_m$, see \eqref{Lmcon}, we have
\begin{align*}
\iln_m(T_*+\theta t)
&=\iln_m((1-\theta)T_*+\theta (T_*+t))\\
&\ge (1-\theta) \iln_m (T_*) + \theta \iln_m(T_*+t). 
\end{align*}

Since $\iln_m(T_*)>0$, it follows that $\iln_m(T_*+\theta t)\ge \theta \iln_m(T_*+t)$, and we have
\beq\label{iine6}
\iln_m(T_*+\theta t)^{-\lambda}\le \theta^{-\lambda}\iln_m(T_*+t)^{-\lambda}. 
\eeq

We obtain \eqref{iine2} from \eqref{iine5}, \eqref{eF2} and \eqref{iine6}.
\end{proof}

\begin{lemma}\label{LL-lem} 
Let $ \sigma >0$, $N \in \N$, and $k_j\in\N$, $\alpha_j\in\R$ for $j=1,2,\ldots, N$.  
Denote 
$$k_*=\max\{k_j:1\le j\le N\}\text{ and let }T_*>E_{k_*}(0).$$

Then, it holds, for all $\lambda\in(0,1)$, that 
\beq\label{log}
\int_0^t e^{-\sigma (t-\tau)} \prod_{j=1}^N (\iln_{k_j}(T_*+\tau))^{\alpha_j} d\tau  =\frac{1}{\sigma} \prod_{j=1}^N (\iln_{k_j}(T_*+t))^{\alpha_j}  +  \bigo(t^{-\lambda}).
\eeq

In particular, for $k\in\N$, $T_*>E_k(0)$, and any $\alpha\in\R$, $\lambda \in (0,1)$, one has 
\beq \label{pow}
\int_0^t e^{-\sigma (t-\tau)}  (\iln_k(T_*+\tau))^\alpha d\tau  =\frac{1}{\sigma} (\iln_k(t+ T_{*}))^\alpha + \bigo(t^{-\lambda}).
\eeq

\end{lemma}
\begin{proof}
Denote $F(t)=\prod_{j=1}^N (\iln_{k_j}(t))^{\alpha_j} $ and $I=\int_0^t e^{-\sigma(t- \tau)} F(T_*+\tau) d\tau$.
Integrating by parts gives
\beq\label{F0}
 I =  \frac{1}{\sigma} \Big( F(T_*+t)  -  F(T_*)  e^{-\sigma t}\Big)
 - \sum_{j=1}^N \frac{1}{\sigma}\int_0^t  e^{-\sigma (t-\tau)}   \frac{\alpha_j F(T_*+\tau)}{(T_*+\tau)\prod_{p=1}^{k_j}\iln_p(T_*+\tau) }
d\tau.
 \eeq
 
Let $\lambda$ be any number in the interval $(0,1)$. Note that
 \beq\label{F1}
 F(T_*) e^{-\sigma t}=\bigo(t^{-\lambda}).
 \eeq

 Also, there exists, for each $j=1,2,\ldots,N$, a positive number  $C_j>0$ such that
 \beqs
\left|\frac{ F(T_*+t)}{(T_*+t)\prod_{p=1}^{k_j}\iln_p(T_*+t) }\right|\le C_j (T_*+t)^{-\lambda},\quad \forall t\ge 0.
 \eeqs
By this estimate and inequality \eqref{iine2} with $m=0$,  
\beq \label{F2}
\int_0^t e^{-\sigma (t-\tau)}   \frac{ F(T_*+\tau)}{(T_*+\tau)\prod_{p=1}^{k_j}\iln_p(T_*+\tau) }
 d\tau=\bigo(t^{-\lambda}).
 \eeq 
Then combining \eqref{F0}, \eqref{F1} and \eqref{F2},  we obtain \eqref{log}. Inequality \eqref{pow} is a special case of \eqref{log} when $N=1$, $k_1=k$, and $\alpha_1=\alpha$.
\end{proof}

Our results will cover more complicated expansions than \eqref{expan}.
They involve the following type of functions which are more general than the polynomials in Definition \ref{polydef}.

\begin{definition}\label{Pdef}
For $z=(z_1,z_2,\ldots,z_k)\in (0,\infty)^k$  and $\alpha=(\alpha_1,\alpha_2,\ldots,\alpha_k)\in \R^k$,  with $k\in\N$, denote 
\beqs
z^\alpha=\prod_{j=1}^k z_j^{\alpha_j}.
\eeqs   

Let $X$ be a real linear space and $k\in \N$. Define $\classP(k,X)$ to be the set of functions $p:(0,\infty)^k\to X$ of the form
\beq\label{Ppz}
p(z)=\sum_{\alpha\in I } c_\alpha z^\alpha\text{ for $z\in (0,\infty)^k$,}
\eeq 
where $I$ is a non-empty, finite subset of $\R^k$, and $c_\alpha$'s are vectors in $X$.
\end{definition}

Below are immediate observations about Definition \ref{Pdef}. 

\begin{enumerate}[label=(\alph*)]
 \item\label{Ca} $\classP(k,X)$ contains all polynomials from $\R^k$ to $X$, in the sense that, if $p:\R^k\to X$ is a polynomial, then its restriction on $(0,\infty)^k$ belongs to $\classP(k,X)$.
 
 \item\label{Cb} Each $\classP(k,X)$ is a linear space.

 \item\label{Cc} Given $m,k\in\N$ with $m>k$. 
 For $z=(z_1,\ldots,z_m)\in\R^m$, denote $\bar z=(z_1,\ldots,z_k)$. Let $p\in \classP(k,X)$.
 Define function $\mathcal I_{k,m}p:(0,\infty)^m\to X$ by 
 \beq \label{Ikm}
 z\in (0,\infty)^m \mapsto (\mathcal I_{k,m}p)(z):=p(\bar z).
 \eeq 
 
 With the standard embedding $\R^k\simeq\R^k\times\{ 0\}^{m-k}\subset \R^m$, we have $\mathcal I_{k,m}p=p$ on $(0,\infty)^k$.
 Hence, $\mathcal I_{k,m}p$ can be seen as an extension of function $p$ from $(0,\infty)^k$ to $(0,\infty)^m$.
 
 Suppose $p$ is given by \eqref{Ppz}. Then, for $ z\in (0,\infty)^m$, 
\beqs
(\mathcal I_{k,m}p)(z)=\sum_{\alpha\in I } c_\alpha \bar z^\alpha=\sum_{(\alpha,0_{m-k})\in J} c_{\alpha} z^{(\alpha,0_{m-k})},
\eeqs
 where  $0_{m-k}$ denotes the zero vector in $\R^{m-k}$, and  $J=I\times\{ 0\}^{m-k}\subset \R^m$, which is non-empty and finite.
 Hence, $\mathcal I_{k,m}p$ belongs to $\classP(m,X)$.

 Clearly, $\mathcal I_{k,m}: \classP(k,X) \to \classP(m,X)$ is linear and one-to-one. Therefore, we have the embedding
\beq \label{Pembed}
\classP(k,X)\simeq \mathcal I_{k,m}(\classP(k,X))\subset \classP(m,X).
\eeq
\end{enumerate}

\begin{lemma}\label{Pprop}
The following statements hold true.
\begin{enumerate}[label={\tnum}]
 \item \label{P1} If $p_j\in\classP(k,X_j)$ for $1\le j\le m$, where $m\ge 1$, $X_j$'s are linear spaces, and $L$ is an $m$-linear mapping from $\prod_{j=1}^m X_j$ to $X$, then
 $L(p_1,p_2,\ldots,p_m)\in\classP(k,X)$. 
 \item \label{P2} If $p\in\classP(k,X)$ and $L:X\to Y$ is a linear mapping, then $Lp\in\classP(k,Y)$.
 \item \label{P3} If $p\in\classP(k,\R^n)$ and $1\le j\le n$, then  the canonical projection $\pi_j p$, that maps $p$ to its $j$-th component, belongs to $\classP(k,\R)$. 
 \item \label{P4} If $p\in \classP(k,\R)$ and $q\in \classP(k,X)$, then the product $p q\in \classP(k,X)$.
 
 Consequently, if $p_j\in \classP(k,\R)$ for $1\le j\le m$, then $p_1 p_2\ldots p_m\in \classP(k,\R)$.
 
 \item \label{P5} If $p\in\classP(k,\R^n)$ and $q$ is a polynomial from $\R^n$ to $X$, then the composition $q\circ p$ belongs to $\classP(k,X)$. 
 \item \label{P6} In case $X$ is a normed space, if $p\in \classP(k,X)$, then so is each partial derivative $\partial p(z)/\partial z_j$, for $z=(z_1,z_2,\ldots,z_k)\in(0,\infty)^k$ and $1\le j\le k$.
\end{enumerate} 
\end{lemma}
\begin{proof}
Within this proof, all summations $\Sigma$ are meant to have finitely many terms.

We prove \ref{P1} first. For $z\in (0,\infty)^k$, we have 
\begin{align*}
 L(p_1(z),p_2(z),\ldots,p_m(z))
 &=L\Big(\sum_{\alpha_1\in \R^k} c^{(1)}_{\alpha_1}z^{\alpha_1},\sum_{\alpha_2\in \R^k} c^{(2)}_{\alpha_2}z^{\alpha_2},\ldots,\sum_{\alpha_m\in \R^k} c^{(m)}_{\alpha_m}z^{\alpha_m}\Big)\\
 &=\sum_{\alpha_1,\alpha_2,\ldots \alpha_m\in \R^k} L( c^{(1)}_{\alpha_1},c^{(2)}_{\alpha_2}, \ldots,c^{(m)}_{\alpha_m})
 z^{\alpha_1}z^{\alpha_2}\ldots z^{\alpha_m}\\
 &=\sum_{\alpha_1,\alpha_2,\ldots \alpha_m\in \R^k} L( c^{(1)}_{\alpha_1},c^{(2)}_{\alpha_2}, \ldots,c^{(m)}_{\alpha_m}) z^{\sum_{j=1}^m \alpha_j}.
 \end{align*}
Therefore, $L(p_1,p_2,\ldots,p_m)\in \classP(k,X)$.

Then \ref{P2} is a consequence of \ref{P1} when $m=1$, and \ref{P3} is a consequence of \ref{P2} with $L=\pi_j$.

For the first part of \ref{P4}, we apply \ref{P1} to $m=2$, $X_1=\R$, $X_2=X$, and $L:\R\times X\to X$ defined by $L(p,q)=p q$.
For the second part, we apply the first part many times with $X=\R$.

We now prove \ref{P5}. Suppose $p=(p_1,p_2,\ldots,p_n)$ and, for $z\in\R^k$, $q(z)=\sum_{\alpha\in Z_+^k}c_\alpha z^\alpha$, where  $c_\alpha\in X$. 
For $z\in (0,\infty)^k$, we have
\begin{align*}
 q(p(z))=\sum_{\alpha\in \Z_+^k} c_\alpha p(z)^{\alpha}
 =\sum_{\alpha_1,\ldots,\alpha_k\in \Z_+} c_\alpha \prod_{j=1}^k p_j(z)^{\alpha_j}.
\end{align*}

By \ref{P3}, each $p_j\in \classP(k,\R)$, and by \ref{P4}, noticing that $\alpha_j\in \Z_+$, $p_j(z)^{\alpha_j}\in \classP(k,\R)$.
By \ref{P4}, again, $\prod_{j=1}^k p_j(z)^{\alpha_j}\in \classP(k,\R)$ and, then, $c_\alpha \prod_{j=1}^k p_j(z)^{\alpha_j}\in \classP(k,X)$.
Summing finitely many such terms yields $p(q(z))\in \classP(k,X)$.

Finally, \ref{P6} is obviously true.
\end{proof}

The next result is similar to Lemma \ref{exp-odelem}, but allows the function $p$ in equation \eqref{odey} to take a more general form.

\begin{lemma}\label{iterlog-odelem}
Given $m\in \Z_+$, $k\in\N$, and a number $t_0>E_{m+k}(0)$.
Let  $p\in \classP(k,\R^n)$, and let function $g\in C([t_0,\infty),\R^n)$ satisfy 
\beq \label{gprop}
|g(t)|=\bigo(\iln_m^{-\alpha}(t))\text{ for some }\alpha>0.
\eeq 

Suppose $y\in C([t_0,\infty),\R^n)$ is a solution of 
 \beq\label{yeq4}
 y'=-Ay+p(\LL_{m,k}(t))+g(t)\text{ on } (t_0,\infty).
 \eeq

 Then there exists  $\delta >0$  such that 
 \beq\label{ypL}
 |y(t)-A^{-1}p(\LL_{m,k}(t))|=\bigo(\iln_m(t)^{-\delta}).
 \eeq
 \end{lemma}
\begin{proof}
In this proof, $\sum_{\alpha}$ denotes a sum over finitely many $\alpha\in \R^k$. 
Let 
\beq\label{pz} p(z)=\sum_{\alpha} c_\alpha z^\alpha \text{ for $z\in (0,\infty)^k$,}\eeq  
 where each $c_\alpha$ is constant vector in $\R^n$.

Let $1\le j\le d$. Applying $R_{\lambda_j}$ to equation \eqref{yeq4} gives
\beqs
 (R_{\lambda_j}y)'=-\lambda_j(R_{\lambda_j} y) +R_{\lambda_j} p(\LL_{m,k}(t)) + R_{\lambda_j} g(t) \text{ on } (t_0,\infty).
 \eeqs
 
 By the variation of constants formula, we obtain, for $t\ge 0$,
 \beq \label{Rly}
 \begin{aligned}
R_{\lambda_j} y(t_0+t)
&= R_{\lambda_j} y(0) e^{-\lambda_j t} + \int_0^t e^{-\lambda_j (t-\tau)} R_{\lambda_j} p(\LL_{m,k}(t_0+\tau)) d\tau\\
&\quad + \int_0^t e^{-\lambda_j (t-\tau)} R_{\lambda_j} g(t_0+\tau) d \tau . 
 \end{aligned}
\eeq

  Obviously, the first term on the right-hand side of \eqref{Rly} is of $\bigo( \iln_m(t_0+\tau)^{-\lambda})$ for any $\lambda>0$.
For the last term on the right-hand side of \eqref{Rly}, using property \eqref{gprop} and then applying Lemma \ref{plnlem}, we have
 \begin{align*}
\Big|\int_0^t e^{-\lambda_j (t-\tau)}R_{\lambda_j} g(t_0+\tau) d \tau \Big|
\le C \int_0^t e^{-\lambda_j(t- \tau)} \iln_m(t_0+\tau)^{-\alpha} d\tau 
=\bigo( \iln_m(t_0+\tau)^{-\alpha}). 
 \end{align*}

Rewriting the second term on the right-hand side of \eqref{Rly} with the use of the explicit form \eqref{pz} of $p(z)$, and applying \eqref{log} of Lemma \ref{LL-lem} give, for any $\gamma\in(0,1)$,
 \begin{align*}
\int_0^t e^{-\lambda_j (t-\tau)}R_{\lambda_j} p(\LL_{m,k}(t_0+\tau)) d \tau 
&=\sum_\alpha \Big(  \int_0^t e^{-\lambda_j(t-\tau)} \LL_{m,k}(t_0+\tau)^\alpha d \tau  \Big) R_{\lambda_j}c_\alpha \\
&= \sum_\alpha   \frac{1}{\lambda_j} \LL_{m,k}(t_0+t)^\alpha R_{\lambda_j} c_\alpha  + \bigo(t^{-\gamma})\\
&= A^{-1} R_{\lambda_j} \Big( \sum_\alpha   c_\alpha  \LL_{m,k}(t_0+t)^\alpha  \Big) + \bigo(\iln_m(t_0+t)^{-\gamma}).
 \end{align*}
 
Combining the above, we obtain
 \begin{align*}
R_{\lambda_j} y(t_0+t)
&= A^{-1} R_{\lambda_j} p(\LL_{m,k}(t_0+t)) + \bigo(\iln_m(t_0+t)^{-\delta})
  \end{align*}
for some $\delta>0$, which implies
 \begin{align*}
R_{\lambda_j} y(t)
&= A^{-1} R_{\lambda_j} p(\LL_{m,k}(t)) + \bigo(\iln_m(t)^{-\delta}). 
  \end{align*}
  
Summing up this equation in $j$ from $1$ to $d$, and using \eqref{Pc} give 
  \begin{align*}
 y(t) &=  A^{-1}p(\LL_{m,k}(t)) + \bigo(\iln_m(t)^{-\delta}),
 \end{align*} 
which proves \eqref{ypL}.
\end{proof}

\begin{corollary}\label{powlog-odelem}
Let $m\in\Z_+$ and $j\in\N$ such that $j > m$, and let $T_*>E_j(0)$. Assume $g$ is a continuous  function from $[T_*,\infty)$ to $\R^n$ satisfying $|g(t)|=\bigo(\iln_m(t)^{-\alpha})$ for some $\alpha>0$.
Let  $p:\R\to\R^n$ be a polynomial, and $y\in C([T_*,\infty),\R^n)$ be a solution of 
 \beqs
 y'=-Ay+p(\iln_j(t))+g(t)\text{ on } (T_*,\infty).
 \eeqs
Then there exists  $\delta >0$  such that 
 \beq\label{yAp}
 |y(t)-A^{-1}p(\iln_j(t))|=\bigo(\iln_m(t)^{-\delta}).
 \eeq
 \end{corollary}
\begin{proof}
Let $k=j-m$, define function $\tilde p(z_1,z_2,\ldots,z_k)=p(z_k)$. Note that $\tilde p\in \classP(k,\R^n)$ and $\tilde p(\LL_{m,k}(t))=p(\iln_j(t))$.
Applying Lemma \ref{iterlog-odelem} with $\tilde p$ replacing $p$, we obtain \eqref{yAp} from \eqref{ypL}.
\end{proof}

\section{Basic existence result and large-time estimates}\label{odesec}

We establish basic facts for equation \eqref{sys-eq}.
First, we have the global existence result for small initial data and forcing function.

\begin{theorem}\label{thmsmall}
There are positive numbers $\varep_0$ and $\varep_1$ such that if $y_0\in\R^n$ with $|y_0|<\varep_0$, and 
$f\in C([0,\infty), \R^n)$ with 
\beq\label{fep} \|f\|_{\infty}:=\sup\{|f(t)|:t\in[0,\infty)\}<\varep_1,
\eeq 
then there exists a unique solution $y\in C^1([0,\infty),\R^n)$ of \eqref{sys-eq} on $[0,\infty)$ with $y(0)=y_0$. 
If, in addition, 
\beq\label{limf} 
\lim_{t\to\infty}f(t)=0,
\eeq  
then 
 \begin{equation}\label{limy}
  \lim_{t\to\infty} y(t)=0.
 \end{equation}
 \end{theorem}
\begin{proof}
(a) First, we consider the special case when $A=A_0$, where $A_0$ is in \eqref{Adiag}.
We have 
\beq\label{Ayy}
(Ay)\cdot y\ge \lambda_1 |y|^2 \quad\forall y\in\R^n.
\eeq

Let $r_*$ and $c_*$ be as in \eqref{Gyy}. Set
\beq\label{C0}
 C_0=  \min\Big\{r_*,\frac{\lambda_1}{4c_*}\Big\} , \quad 
\varep_0= \min\Big \{\frac{C_0}{2}, c_*\Big\},\quad \text{and}\quad
\varep_1=\frac{\lambda_1 C_0}{2\sqrt{2}}.
\eeq

Suppose $|y_0|<\varep_0$ and $\|f\|_\infty<\varep_1$.
Note that $|y_0|<C_0$.
By the the local existence and uniqueness theory, see e.g. \cite{CL55,JHale78}, there exists a maximal $T\in(0,\infty]$ such  that there is a unique solution $y(t)$ on $[0,T)$ that satisfies  
\beq \label{ysmall}
|y(t)|<C_0 \text{ on $[0,T)$.}
\eeq 

We claim $T=\infty$.
Suppose this not true, then, by \eqref{ysmall}, the local existence result and the maximality of $T$, the solution $y(t)$ exists on $[0,T')$ for some $T'>T$ and we have
\beq\label{yt0}
|y(T)|=C_0.
\eeq

Let $\varep>0$ be arbitrary. Taking the dot product of \eqref{sys-eq} with $y(t)$, then using property \eqref{Ayy}, 
Cauchy-Schwarz's inequality, and estimate \eqref{Gyy}, we obtain, for $t\in(0,T)$,
\beqs
\frac12\ddt|y|^2 
= - Ay\cdot y +G(y)\cdot y +f\cdot y
\le -\lambda_1 |y|^2 + c_* |y|^{2}|y| + |f|\cdot |y|.
\eeqs
Applying Cauchy's inequality to the last product yields
\beq\label{d0}
\frac12\ddt|y|^2 
\le -(\lambda_1  - c_* |y| - \varepsilon ) |y|^2+\frac{|f|^2}{4\varepsilon}  \text{ on }(0,T).
\eeq

Taking $\varep=\lambda_1/2$ in \eqref{d0}, and utilizing estimate \eqref{ysmall} for the first $|y|$ on its right-hand side, for $t\in(0,T)$, we have that
\beqs
\frac d{dt}|y|^2 
\le -(\lambda_1-2 c_* C_0)|y|^2 + \frac{|f|^2}{\lambda_1}  
\text{ on }(0,T).
\eeqs

Set $\alpha_0=\lambda_1/2>0$.
With $C_0$ defined in \eqref{C0}, we have $\lambda_1-2 c_* C_0 \ge \alpha_0$, and, hence, 
\beq\label{d2}
\frac d{dt}|y|^2 
\le -\alpha_0 |y|^2 +  \frac{|f|^2}{\lambda_1} \text{ on }(0,T).
\eeq

By Gronwall's inequality and assumption \eqref{fep}, we have for  $t\in[0,T)$ that 
\beqs
|y(t)|^2 \le |y_0|^2 e^{-\alpha_0 t} +\frac{ \varep_1^2}{\lambda_1 \alpha_0}\le \varep_0^2  +\frac{\varep_1^2}{\lambda_1 \alpha_0} \le \frac{C_0^2}2.
\eeqs 

Letting $t\to T^-$, we obtain 
$y(T)\le C_0/{\sqrt{2}}$, which contradicts \eqref{yt0}.
Therefore $T=\infty$.

\medskip
Now, assume \eqref{limf}. We have \eqref{d2} holds for all $t>0$.
Applying \cite[Lemma 3.9]{HI2} to \eqref{d2} gives
\beqs
0\le \limsup_{t\to\infty} |y(t)|^2\le \limsup_{t\to\infty} \frac{|f(t)|^2}{\alpha_0 \lambda_1}=0.
\eeqs
This proves \eqref{limy}.

(b) Consider the general case of $A$ as in \eqref{Adiag}.
We make use of the transformations 
 \beq\label{trans}
z=Sy,\quad  \tilde G(z)=SG(S^{-1}z),\quad \tilde f(t)=Sf(t).
 \eeq

Then equation \eqref{sys-eq} is equivalent to 
 \beq\label{zeq}
 z'=-A_0 z + \tilde G(z) + \tilde f(t).
 \eeq

There exists a constant $c\ge 1$ such that 
 \beq\label{zf}
c^{-1}|y|\le |z|\le c|y|\text{ and } c^{-1}|f|\le |\tilde f|\le c |f|.
 \eeq
 
Thanks to the equivalent norms in \eqref{zf}, $\tilde G$ and $\tilde f$ have similar properties to $G$ and $f$, respectively.
By  part (a) applied to equation \eqref{zeq}, we obtain the results for $z(t)$ and $\tilde f(t)$.
Then, thanks to relation \eqref{zf} again,  the results for $y(t)$ and $f(t)$ follow. We omit the details.
\end{proof}

There are examples of $G(y)$ such that the global solution $y(t)$ exists even when $|y_0|$ is not small.
Theorem \ref{thmsmall} only guarantees that the set of global solutions of our interest is not empty.
Certainly, not all solutions decay to zero. For example, even when $f=0$, the system \eqref{sys-eq} may have a non-zero steady state.

When $f(t)$ has more specific decay than \eqref{limf}, then we can obtain corresponding large-time estimates for $y(t)$.

\begin{theorem}\label{thmdecay}
 Let $\psi(t)=e^t$ or $\psi(t)=\iln_k (t) $ for some integer $k\ge 0$.
 Assume there is $T\ge 0$ such that $f\in C((T,\infty))$ and 
 \beq\label{fbound} 
 |f(t)|=\bigo(\psi(t)^{-\alpha})\text{ for some }\alpha >0.
 \eeq 
 
 Let $y(t)$ be a solution of \eqref{sys-eq} on $(T,\infty)$ and satisfy 
 \beq\label{limyy}
 \liminf_{t\to\infty}|y(t)|<\varep_0,
 \eeq  
 where $\varep_0$ is as in Theorem \ref{thmsmall}.
 Then
\beq\label{larg-est}
y(t)=
\begin{cases}
 \bigo(\psi(t)^{-\min \{\lambda_1,\alpha \}}),    &\text{ if } \psi(t)=e^t,\ \alpha\ne \lambda_1,\\
  \bigo(t\psi(t)^{-\alpha}),  &\text{ if } \psi(t)=e^t,\ \alpha = \lambda_1,\\
 \bigo(\psi(t)^{-\alpha}), &\text{ if } \psi(t)=\iln_k(t).
\end{cases}
\eeq 

 \end{theorem}
\begin{proof}
Similar to the proof of Theorem \ref{thmsmall}, by using the transformations in \eqref{trans} and property \eqref{zf}, we can assume $A$ takes the form \eqref{Adiag}. Note in this case that 
\beq\label{etA} 
\|e^{-tA}\|= e^{-\lambda_1 t}\text{ for all $t\ge 0$.}
\eeq

Let $\varep_1$ be as in Theorem \ref{thmsmall}. By \eqref{limyy} and \eqref{fbound}, there exists $t_*>T$ such that
\beqs
|y(t_*)|<\varep_0\text{ and } \sup\{|f(t)|:t\ge t_*\}<\varep_1.
\eeqs

Note also that $f$ satisfies \eqref{limf}. Applying Theorem \ref{thmsmall} yields \eqref{limy}.

Denote $\mu=\min\{\lambda_1,\alpha \}$. Let $c_*$ be as in \eqref{Gyy}, and $\delta$ be an arbitrary positive number  with $\delta<\min\{2\varep_0 c_*,\mu\}$. 
By assumption \eqref{fbound} and the proved fact \eqref{limy}, there exist $t_0>t_*$ and $C>0$ such that $y\in C^1([t_0,\infty),\R^n)$ and
\beq\label{yft} |y(t)|\le \frac\delta{2c_*}<\varep_0, \quad |f(t)|\le C\psi(t)^{-\alpha}<\varep_1 \text{ for all $t \ge t_0$.} 
\eeq 

Then \eqref{d0}, with the choice $\varep=\delta/2$, and the estimates in \eqref{yft} yield 
\begin{align*}
\ddt|y(t_0+t)|^2 
&\le -2(\lambda_1 - \delta/2-\delta/2 ) |y(t_0+t)|^2+ C^2\delta^{-1}  \psi(t_0+t)^{-2\alpha} \\
&= -2(\mu - \delta ) |y(t_0+t)|^2+ C^2\delta^{-1}  \psi(t_0+t)^{-2\alpha}
\end{align*}
for all $t>0$. By Gronwall's inequality, we have, for $t>0$,
\begin{align*}
|y(t_0+t)|^2 &\le |y(t_0)|^2e^{-2(\mu -\delta )t} + C^2\delta^{-1} \int_0^t e^{-2(\mu -\delta )(t-\tau)} \psi(t_0+\tau)^{-2\alpha} d\tau.
\end{align*}

Estimating  the last integral, either by using \eqref{ee} in case $\psi(t)=e^t$, noticing that $\alpha>\mu-\delta$, or 
by applying Lemma \ref{plnlem} in case $\psi(t)=\iln_k(t)$, we obtain 
\beq\label{yshift}
|y(t_0+t)|^2 
=\begin{cases}
 \bigo(\psi(t_0+t)^{- 2(\mu -\delta)}) &  \text{ if } \psi(t)=e^t,\\
 \bigo(\psi(t_0+t)^{-2\alpha})& \text{ if } \psi(t)=\iln_k(t).
 \end{cases}
\eeq

Then the last estimate in \eqref{larg-est} follows from the second estimate in \eqref{yshift}.

Consider $\psi(t)=e^t$ now. Choose, additionally, $\delta<\mu/4$ such that $2(\mu-\delta)\ne \lambda_1$.

For any $T_0\ge t_0$, and $t\ge 0$, by variation of constant formula,
\beqs
y(T_0+t)=e^{-tA}y(T_0)+\int_0^t e^{-(t-\tau)A}\{G(y(T_0+\tau))+f(T_0+\tau)\}d\tau.
\eeqs

Selecting $T_0$ sufficiently large, using \eqref{etA}, \eqref{Gyy}, the first estimate in \eqref{yshift}, and \eqref{fbound}, we obtain
\beqs
|y(T_0+t)|\le e^{-\lambda_1 t}|y(T_0)|+ C'\int_0^t e^{-\lambda_1(t-\tau)}(e^{-2(\mu-\delta)\tau}+e^{-\alpha \tau})d\tau
\eeqs
for some $C'>0$. Note, by using \eqref{ee}, that
\beqs
\int_0^t e^{-\lambda_1(t-\tau)}e^{-2(\mu-\delta)\tau}d\tau
=\bigo(e^{-\min\{\lambda_1,2(\mu-\delta)\}t})
=\bigo(e^{-\min\{\lambda_1,3\mu/2\}t})
=\bigo(e^{-\mu t}).
\eeqs

Then estimating $\int_0^t e^{-\lambda_1(t-\tau)} e^{-\alpha \tau}d\tau$ by using \eqref{ee} again, we obtain 
\beqs
y(T_0+t)=
\begin{cases}
 \bigo(e^{-\mu t}),    &\text{ if } \alpha\ne \lambda_1,\\
  \bigo(t e^{-\mu t}),  &\text{ if } \alpha=\lambda_1.
  \end{cases}
\eeqs 

By shifting time $T_0+t$ back to $t$, we obtain the first two estimates in \eqref{larg-est}.
\end{proof}

\section{Expansions with one secondary base function}\label{basetwo}

In this section, we study the asymptotic expansions of the form \eqref{expan} in Definition \ref{psi-phi}. They are expressed in terms of one primary base function $\psi$ and one secondary base function $\phi$.

We explain Definition \ref{psi-phi} further now.

\begin{enumerate}[label=(\alph*)]
\item The first limit in \eqref{phs1} obviously follows the second one and relation \eqref{compcond}. Nonetheless, we stated both limits in \eqref{phs1} to make  a clear presentation.

\item It follows \eqref{phs1} and \eqref{compcond} that
 \beq\label{ratiolim}
 \lim_{t\to\infty} \frac{\phi(t)^\lambda}{\psi(t)^\alpha}=0\text{ for all } \lambda\in \R,\ \alpha>0.
 \eeq
 
\item  If $g(t)=\sum_{k=1}^N p_k(\phi(t)) \psi(t)^{-\gamma_k}$, then clearly the expansion \eqref{finex} holds.

 \item Consider \eqref{finex}. Suppose $(\gamma_k)_{k=1}^N$ can be extended as to a sequence $(\gamma_k)_{k=1}^\infty$ as in (i). Set $p_k=0$ for $k>N$.
 Then, thanks to \eqref{allam}, we obtain \eqref{expan}. This means the finite expansion \eqref{finex} is a special case of \eqref{expan}. Such an extension is, of course, not unique, and some may be more appropriate than the others, see, e.g., Scenarios \ref{scen2} and \ref{scen3} below.
 
 \item Another type of asymptotic expansions in \cite[Definition 4.1]{CaH2} is
 \beq\label{gold}
 g(t)\sim \sum_{k=1}^\infty \xi_k \psi_k(t),
 \eeq 
where $\xi_k$'s are constant vectors, and $(\psi_k)_{k=1}^\infty$ is a general system of decaying functions.

On the one hand, $p_k(\phi(t))$ in \eqref{expan} is more general than $\xi_k$ in \eqref{gold}.
On the other hand, the $\psi_k$ in \eqref{gold} is more general than $\psi^{-\gamma_k}$ in \eqref{expan}.
\end{enumerate}

It turns out that the asymptotic expansion \eqref{expan}, for any given function $g(t)$, is unique.
This is a direct consequence of the following lemma.

\begin{lemma}\label{ue}
Let the normed space $(X,\|\cdot\|_X)$ and  functions $\psi$, $\phi$ be as in Definition \ref{psi-phi}. Given a function $g:(T',\infty)\to X$ for some $T'\in\R$. Let $0\le \gamma_1<\ldots<\gamma_N$ for some $N\in \N$.
 Suppose $p_1,\ldots,p_N:\R\to X$ are polynomials such that  \eqref{gdrem} holds for some $\mu>\gamma_N$.
Then such polynomials $p_1,\ldots,p_N$ are unique.
\end{lemma}
\begin{proof}
Suppose $\hat p_k$'s are $X$-valued polynomials, for $1\le k\le N$, such that  
\beqs 
\Big\|g(t)-\sum_{k=1}^N \hat p_k(\phi(t))\psi(t)^{-\gamma_k}\Big\|_X=\bigo(\psi(t)^{-\hat\mu})\quad\text{ for some }\hat\mu>\gamma_N.
\eeqs 

 For each $k$, let $h_k=p_k-\hat p_k$, which is an $X$-valued polynomial. By the triangle inequality, we have
\beq\label{tri}
\Big\|\sum_{k=1}^N h_k(\phi(t))\psi(t)^{-\gamma_k}\Big\|_X
\le \Big\|\sum_{k=1}^N p_k(\phi(t))\psi(t)^{-\gamma_k}-g(t)\Big\|_X +\Big\|g(t)-\sum_{k=1}^N \hat p_k(\phi(t))\psi(t)^{-\gamma_k}\Big\|_X,
\eeq
hence
\beqs 
\Big\|\sum_{k=1}^N h_k(\phi(t))\psi(t)^{-\gamma_k}\Big\|_X=\bigo(\psi(t)^{-\mu_*}),\quad\text{ where }\mu_*=\min\{\mu,\hat\mu\}>\gamma_N.
\eeqs

Multiplying this equation by $\psi(t)^{\gamma_1}$ and making use of the property \eqref{ratiolim}, we have 
\beq\label{h1}
\lim_{t\to\infty} h_1(\phi(t))=0.
\eeq

Because $h_1$ is a polynomial, together with \eqref{h1} and the fact $\phi(t)\to\infty$ as $t\to\infty$, we deduce $h_1=0$.
Repeating this argument for the remaining $h_k$'s, we obtain $h_k=0$ for all $k=1,\ldots,N$.
\end{proof}

As a side note, the power $\gamma_1$ in Lemma \ref{ue} is in $[0,\infty)$, which is more general than the positive $\gamma_1$ in expansion \eqref{expan}. This small alteration  aims to provide a short argument at the beginning of the proofs of Lemmas \ref{polylem} and \ref{exp-odelem}.

Of course, there are many choices of $(\psi,\phi)$. We will develop our theory for the following three typical cases.

\begin{definition}\label{typedef} 
We define three types of pair of functions $(\psi,\phi)$.
 
Type 1: $(\psi,\phi)=(e^t,t)$, 

Type 2: $(\psi,\phi)=(t,\ln t)$, and 

Type 3: $(\psi,\phi)=(\iln_{m_*}(t),\iln_{n_*}(t))$, with $n_* > m_* \ge 1$.
\end{definition}

Clearly, the functions $(\psi,\phi)$ in Definition \ref{typedef} satisfy conditions \eqref{phs1} and \eqref{compcond} in Definition \ref{psi-phi}.
For the rest of this section, $(\psi,\phi)$ is one of the three types in Definition  \ref{typedef}.

We note that
\beq\label{dev}
(\psi'(t),\phi'(t))=\begin{cases}
               (\psi(t),1),&\text{ for Type 1},\\
               (1,\psi(t)^{-1}),&\text{ for Type 2},\\
               (\bigo(\psi(t)^{-\alpha}),\bigo(\psi(t)^{-\beta})) \quad\forall \alpha,\beta>0, &\text{ for Type 3}.
              \end{cases}
\eeq

We now focus on the differential equation of our interest -- equation \eqref{sys-eq}. We need a basic requirement on its forcing function $f(t)$.

\begin{assumption}\label{assumpf}
There exists a number $T_f\ge 0$ such that $f$ is continuous on $[T_f,\infty)$.
\end{assumption}

More specific conditions on $f$ will be specified later for each result.  

\begin{definition}\label{mset}
Let $S$ be a subset of $\R$. 

We say $S$ preserves the addition if $x+y\in S$ for all $x,y\in S$.

We say $S$ preserves the unit increment if $x+1\in S$ for all $x\in S$.
\end{definition}

The main assumption on $f$ for this section is the following.

\begin{assumption} \label{muSas}
The function $f(t)$ admits the asymptotic expansion, in the sense of Definition \ref{psi-phi} with $X=\R^n$ and $\|\cdot\|_X=|\cdot|$,
\beq\label{fmu}
f(t) \sim \sum_{k=1}^\infty p_k(\phi(t)) \psi(t)^{-\mu_k},
\eeq
where $p_k$'s are $\R^n$-valued polynomials, $(\mu_k)_{k=1}^\infty$ is a divergent, strictly increasing sequence of positive numbers.
Moreover, the set $\mathcal S:=\{\mu_k:k\in\N\}$ satisfies
\begin{enumerate}[label=\rnum]
 \item\label{Aa}  $\mathcal S$ preserves the addition. 

 \item\label{Ab} In case of Type 1, $\mathcal S$ contains all eigenvalues of $A$.

 \item\label{Ac} In case of Type 2, $\mathcal S$ preserves the unit increment.
\end{enumerate}
\end{assumption}

Note from (b) of Assumption \ref{muSas} that, in case of Type 1, one has
\beq\label{ml1}
\mu_1\le \lambda_1.
\eeq

Below, we discuss typical scenarios when Assumption \ref{muSas} holds.

\begin{scenario}\label{scen1} The forcing function $f(t)$ in \eqref{sys-eq} has the following expansion, in the sense of Definition \ref{psi-phi},
\beq \label{f-exp}
f(t) \sim \sum_{k=1}^\infty \hat p_k(\phi(t)) \psi(t)^{-\alpha_k} \text{ in }\R^n,
\eeq
where $\hat p_k$'s are polynomials from $\R$ to $\R^n$, and
$(\alpha_k)_{k=1}^\infty$ is a divergent, strictly increasing sequence of  positive numbers.

For Type 1, let $\mathcal S$ be the joint additive semigroup generated by both $\lambda_j$'s and $\alpha_j$'s, i.e.,
\beq\label{S1}
\mathcal S=\left \{ \sum_{j=1}^k \lambda_{s_j}+ \sum_{j=1}^m \alpha_{\ell_j}:  k,m\in \Z_+,\ k^2+m^2>0,
\ 1\le s_j\le d,\ \ell_j\in \N \right \}.
\eeq

For Type 2, let $\mathcal S$ be defined by
\beq\label{S2}
\mathcal S=\left \{ k+\sum_{j=1}^m \alpha_{\ell_j}: k\in\Z_+,\ m\in\N,\ \ell_j\in \N \right \}.
\eeq

For Type 3, let $\mathcal S$ be the additive semigroup generated by $\alpha_k$'s, i.e.,
\beq\label{S3}
\mathcal S=\left \{ \sum_{j=1}^m \alpha_{\ell_j}: m\in\N,\ \ell_j\in \N \right \}.
\eeq

Re-arrange the set $\mathcal S$ as a sequence 
\beq\label{Sseq} \mathcal S=(\mu_k)_{k=1}^\infty \text{ of strictly increasing positive numbers.}
\eeq

Then $(\mu_k)_{k=1}^\infty$ and $\mathcal S$ satisfy the properties in Assumption \ref{muSas}.
Note that the set $\mathcal S$ in \eqref{S1}, \eqref{S2}, \eqref{S3}  contains the sequence $(\alpha_k)_{k=1}^\infty$. 
Therefore, we can rewrite expansion \eqref{f-exp}, after re-indexing $\hat p_k$'s,  as the expansion \eqref{fmu}.
For example, the first term in expansion \eqref{fmu} is identified by 
\beqs
\mu_1=\begin{cases}
       \min\{\lambda_1,\alpha_1\},&\text{ for Type 1},\\
       \alpha_1,&\text{ for Types 2 and 3},
      \end{cases}
\eeqs 
\beqs 
      p_1=\begin{cases}
       0, &\text{ for Type 1 with $\lambda_1<\alpha_1$},\\
       \hat p_1,&\text{ for Type 1 with $\lambda_1\ge \alpha_1$, and for Types 2 and 3}.
      \end{cases}
\eeqs
\end{scenario}

\begin{scenario}\label{scen2}
In case of finite approximation, as in Definition \ref{psi-phi}(ii),
\beq\label{ffin}
f(t) \sim \sum_{k=1}^N \hat p_k(\phi(t)) \psi(t)^{-\alpha_k} \text{ in }\R^n,
\eeq
then, corresponding to Type 1, 2, 3, the set $\mathcal S$ is defined by formulas \eqref{S1}, \eqref{S2}, \eqref{S3} with restriction $1\le \ell_j\le N$. This set $\mathcal S$ is still infinite, can be arranged as sequence $(\mu_k)_{k=1}^\infty$ as in Scenario \ref{scen1}, and contains  $(\alpha_k)_{k=1}^N$. Hence, again, we can rewrite \eqref{ffin} as \eqref{fmu}.
\end{scenario}

\begin{scenario}\label{scen3}
Function $f$ decays faster than any exponential functions, i.e., $e^{\alpha t} f(t)\to 0$ as $t\to\infty$ for any $\alpha>0$. (This includes the case $f=0$.) Then we only consider Type 1, and let $\mathcal S$ be the semigroup generated by the spectrum of $A$, i.e.,
\beqs
\mathcal S=\left \{ \sum_{j=1}^k \lambda_{s_j}:  k\in\N,\ 1\le s_j\le d \right \}.
\eeqs
Same as in Scenarios \ref{scen1} and \ref{scen2}, we can write \eqref{fmu} with $p_k=0$ for all $k\in\N$.
\end{scenario}

Regarding expansion \eqref{fmu}, denote $\tilde p_k(t)=p_k(\phi(t))$, and  
\beq \label{fkdef}
f_k(t)=\tilde p_k(t) \psi(t)^{-\mu_k }\quad\text{and}\quad
\bar f_N(t)=\sum_{k=1}^N f_k(t).
\eeq 
Then, according to Definition \ref{psi-phi}, for any $N\in\N$, one has
\beq \label{fN}
|f(t)-\bar f_N(t)|
=\Big|f(t)-\sum_{k=1}^N f_k(t)\Big|=\bigo(\psi(t)^{-\mu_{N}-\varep_N})\text{ for some } \varep_N>0.
\eeq

 The type of solutions of equation \eqref{sys-eq} that will be the subject of our analysis is the following. 

\begin{assumption}\label{assumpu}
There exists a number $T_0\ge 0$ such that $y\in C^1((T_0,\infty))$ is a solution of \eqref{sys-eq} on $(T_0,\infty)$,  and $y(t)\to 0$ as $t\to\infty$.
\end{assumption}

Our first main result on the asymptotic expansion for solutions of \eqref{sys-eq} is the next theorem.
In this theorem, we denote
\beqs
\epsilon=
 \begin{cases}
  1, \quad \text{ for Type 1,}\\
 0, \quad \text{ for Types 2 and 3.}
 \end{cases}
 \eeqs

 \begin{theorem}\label{mainthm}
Let Assumptions \ref{assumpf} and \ref{muSas} hold. Let $y(t)$ be a solution of \eqref{sys-eq} as in Assumption \ref{assumpu}.
  Then there exist unique polynomials $q_k:\R\to \R^n$, for $k\in\N$, such that $y(t)$ admits the expansion, in the sense of Definition \ref{psi-phi},
  \beq \label{ymu}
  y(t)\sim \sum_{k=1}^\infty q_k(\phi(t)) \psi(t)^{-\mu_k}.
  \eeq 
 Moreover,  the polynomials $q_k$'s solve, on $\R$, the following equations:
\begin{align} \label{pq1}    
Aq_1 + \epsilon (q_1'-\mu_1 q_1)  &= p_1,\\
\intertext{and, for $k \ge 2$,}
\label{pq-k}    
\begin{split}
 Aq_k + \epsilon (q_k'-\mu_k q_k)  
 &=\sum_{m\ge 2}\ \sum_{\mu_{j_{1}}+\mu_{j_{2}}+\ldots \mu_{j_{m}}=\mu_k} \mathcal G_m(  q_{j_{1}}, q_{j_{2}},\ldots ,  q_{j_{m}}) \\
&\quad +p_k + \chi_k, 
\end{split}
\end{align}
 where $\chi_k:\R\to\R^n$ is a polynomial defined, in cases of Types 1 and 3, by $\chi_k=0$, and,
in case of Type 2, by
\beq\label{chik}
 \chi_k=
 \begin{cases}
\mu_\lambda q_\lambda-q'_\lambda,&\text{if there exists } \lambda \in [1,k-1] \text{ such that } \mu_\lambda +1 = \mu_k, \\
 0,&\text{otherwise}.
 \end{cases}
 \eeq 
 \end{theorem}
 
 We provide some explanations to the theorem above.
 
\begin{enumerate}[label={\rnum}]
 \item\label{Ta}  In \eqref{pq-k}, it is understood that 
 $j_1,j_2,\ldots,j_m\ge 1$.
Also, if the set of indices is empty, then the sum is understood to be zero.
 
\item\label{Tb} Consider the double summation in \eqref{pq-k}. Since $m\ge 2$ and $\mu_{j_i}>0$, we have each $\mu_{j_i}<\mu_k$, hence $j_i\le k-1$. Therefore, those terms $q_{j_i}$'s in \eqref{pq-k} come from the previous step.
 
\item\label{Tc}  For any numbers $M\ge \mu_k/\mu_1$ and $Z\ge k-1$, one has
 \beq\label{sumequiv}
\sum_{2\le m\le M}\sum_{\substack{{1\le j_{1},j_{2},\ldots,j_{m}\le Z}\\ \mu_{j_{1}}+\mu_{j_{2}}+\ldots \mu_{j_{m}}=\mu_k}}
=\sum_{m\ge 2}\sum_{\mu_{j_{1}}+\mu_{j_{2}}+\ldots \mu_{j_{m}}=\mu_k}.
\eeq

We quickly verify \eqref{sumequiv}. Since, obviously, the sum on the left-hand side is part of sum on the right-hand side, it suffices to show the reverse.
Consider the right-hand side of \eqref{sumequiv}.
Firstly, due to \ref{Tb}, we have $j_{i}\le k-1\le Z$. 
Secondly, we have
$m\mu_1\le \sum_{i=1}^m \mu_{j_i}=\mu_k,$ which yields $m\le \mu_k/\mu_1\le M$.

Therefore, thanks to \eqref{sumequiv}, the double summation in \eqref{pq-k}, in fact, is a finite sum.

\item\label{Td} For our convenience, define  
 \beq\label{chi1} 
 \chi_1=0 \text{ as a function from $\R$ to $\R^n$.}
 \eeq

 Consider formula \eqref{pq-k} even for $k=1$. Then there are no indices to satisfy the conditions in the double sum on the right-hand side of \eqref{pq-k}. Hence, by convention in \ref{Ta}, it is  $0$. Therefore, \eqref{pq-k} for $k=1$, in fact, reads as \eqref{pq1}.

\item\label{Te} With the observation in \ref{Td}, we can combine \eqref{pq1} and \eqref{pq-k} into
\beq\label{pqall}    
 Aq_k + \epsilon (q_k'-\mu_k q_k)  
=\sum_{m\ge 2}\ \sum_{\mu_{j_{1}}+\mu_{j_{2}}+\ldots \mu_{j_{m}}=\mu_k} \mathcal G_m(  q_{j_{1}}, q_{j_{2}},\ldots ,  q_{j_{m}}) +p_k + \chi_k,
\eeq
for  all $k\in\N$.

\item\label{Tf} The index $\lambda$ in \eqref{chik}, if exists, is obviously unique. Thus, $\chi_k$ is well-defined.
\end{enumerate}

 \begin{proof}[Proof of Theorem \ref{mainthm}]
 The uniqueness of polynomials $q_k$'s comes from Lemma \ref{ue}.
 It remains to establish \eqref{ymu}--\eqref{pq-k}.
 
For $N\in\N$, denote by ($\mathcal T_N$) the statement: 
\textit{The equation \eqref{pqall} holds true on $\R$, for $k=1,2,\ldots,N$, and
\beqs
\Big|y(t)-\sum_{k=1}^N q_k(\phi(t))\psi(t)^{-\mu_k} \Big|=\bigo(\psi(t)^{-\mu_N-\delta_N})
\text{ for some $\delta_N>0$. }
\eeqs
}

We will find, by induction, polynomials $q_k:\R\to\R^n$, for $k\in\N$, such that ($\mathcal T_N$) holds true for all $N\in\N$.
In the calculations below, $t$ will be sufficiently large.

\medskip
\textbf{First step.} Let  $N=1$.
Note, by the triangle inequality, \eqref{fN}, \eqref{fkdef} and \eqref{ratiolim}, that
\beqs
|f(t)| \le |f(t)- f_1(t)| + |f_1(t)| =\bigo(\psi(t)^{-\mu_1 - \varep_1}) + \bigo(\psi(t)^{-\mu_1+\delta})
\eeqs
for all $\delta>0.$ This yields
\beq\label{ffirst} 
f(t)= \bigo( \psi(t)^{-\mu_1+\delta}) \quad \forall \delta>0.
\eeq 
Applying Theorem \ref{thmdecay}, with the use of property \eqref{ffirst}, gives
\beq \label{u-first}
y(t)= \bigo( \psi(t)^{-\mu_1+\delta}) \quad \forall \delta>0.
\eeq

Let $w_0(t)=\psi(t)^{\mu_1}y(t)$. Then
 \beqs
 w_0'
 =\psi ^{\mu_1}y' +{\mu_1}\psi ^{\mu_1-1}\psi' y 
 =\psi ^{\mu_1}(-Ay + G(y)+f)+{\mu_1}\psi ^{\mu_1-1}\psi' y.
 \eeqs 
 
Thus,
 \beq \label{w00}
 w_0'
=-Aw_0 + \psi ^{\mu_1}G(y)+\psi ^{\mu_1}(f-f_1 )+\psi ^{\mu_1}f_1 +{\mu_1}\psi ^{\mu_1-1}\psi' y .
 \eeq 

Now fix $\delta>0$ such that $\delta<\min\{1,\mu_1/2\}$.
By  \eqref{Gyy} and \eqref{u-first} we have 
  \beq\label{w01}
  \psi(t)^{\mu_1}|G(y(t))|
  =\psi(t)^{\mu_1}\bigo(\psi(t)^{-2\mu_1 + 2\delta})
  =\bigo(\psi(t)^{-(\mu_1 - 2\delta)}).
 \eeq
 
By \eqref{fN},
  \beq\label{w02}
  \psi(t)^{\mu_1}|f(t)-f_1(t)|=\psi(t)^{\mu_1}\bigo(\psi(t)^{-\mu_1 - \varep_1})= \bigo(\psi(t)^{-\varep_1}).
 \eeq
 
By formula of $\psi'(t)$ in \eqref{dev} and estimate \eqref{u-first}, 
\beqs
  \mu_1\psi(t)^{\mu_1-1}\psi'(t)y(t)=
\begin{cases}
\mu_1 w_0(t) &\text{ for Type 1},\\
\psi(t)^{\mu_1-1}\bigo( \psi(t)^{-\mu_1+\delta})&\text{ for Type 2},\\
\psi(t)^{\mu_1-1}\bigo(\psi(t)^{-\lambda})\bigo( \psi(t)^{-\mu_1+\delta}), \forall \lambda>0,&\text{ for Type 3}.
\end{cases}
\eeqs
Thus,
\beq\label{w03}
  \mu_1\psi(t)^{\mu_1-1}\psi'(t)y(t)=
\begin{cases}
\mu_1 w_0(t) &\text{ for Type 1},\\
\bigo(\psi(t)^{-(1-\delta)})&\text{ for Types 2 and 3}.
\end{cases}
\eeq

Combining \eqref{w00}, \eqref{w01}, \eqref{w02}, \eqref{w03} with the fact $\psi(t)^{\mu_1}f_1(t)=p_1(\phi(t))$, we arrive at
 \beq\label{woeq}
 w_0'(t) =-(A-\epsilon\mu_1 I_n)w_0(t) +  p_1(\phi(t)) + \bigo(\psi(t)^{-\widehat\delta_1}),
 \eeq
 where $\widehat\delta_1=\min\{\mu_1-2\delta,\varep_1,1-\delta\}>0$.
 
In case of Type 1, we apply Lemma \ref{exp-odelem} to equation \eqref{woeq} taking $\lambda=\mu_1$. Thanks to \eqref{ml1}, we have $\lambda\le \lambda_1$, hence there is no need to check condition \eqref{ellams}.

In case of Types 2 and 3, we apply Corollary \ref{powlog-odelem} to equation \eqref{woeq} noticing that $\epsilon=0$.

Then there exist a polynomial $q_1:\R\to\R^n$ and a number $\delta_1 >0$ such that
\eqref{pq1}, which is \eqref{pqall} for $k=1$, holds on $\R$, and
\beqs
  |w_0(t)-q_{1}(\phi(t))|= \bigo(\psi(t)^{-\delta_1 }).
\eeqs
Multiplying the last equation by $ \psi(t)^{-\mu_1 }$ yields 
 \beqs
|y(t)-q_1(\phi(t)) \psi(t)^{-\mu_1 } |=\bigo(\psi(t)^{-\mu_1 - \delta_1 }).
\eeqs
Therefore, statement ($\mathcal T_1$) holds true.

\medskip
\textbf{ Induction step.} Let  $N\ge 1$.
Suppose there are polynomials $q_k$'s, for $1\le k\le N$, such that the statement ($\mathcal T_N$) is true.

For $k=1,2,\ldots,N$, let
$\tilde q_k(t)=q_k(\phi(t))$, and 
\beq \label{ys}
y_k(t)=\tilde q_k(t)\psi(t)^{-\mu_k},\quad 
\bar y_N(t)=\sum_{k=1}^N y_k(t),\quad 
v_N(t)=y(t)-\bar y_N(t).
\eeq 

By the induction hypothesis, we have
\beq\label{vNrate}
v_N(t)=\bigo(\psi(t)^{-\mu_N-\delta_N})\text{ for  some }\delta_N>0.
\eeq

Note from \eqref{ratiolim} that 
\beqs
|y_k(t)|=\bigo(\psi(t)^{-\mu_k+\delta}),\quad\forall \delta>0.
\eeqs
Then
\beq\label{yNbrate}
|\bar y_N(t)|\le \sum_{k=1}^N |y_k(t)|=\bigo(\psi(t)^{-\mu_1+\delta}),\quad\forall \delta>0.
\eeq

(a) Let $w_N(t)=\psi(t)^{\mu_{N+1}}v_N(t)$. We write an appropriate differential equation for $w_N(t)$. We have
\begin{align*}
w_N'&=\mu_{N+1}\psi^{\mu_{N+1}-1}\psi'  v_N + \psi^{\mu_{N+1}} (y'-\sum_{k=1}^N y_k')\\
&=\mu_{N+1}\psi^{\mu_{N+1}-1}\psi' v_N  + \psi^{\mu_{N+1}}(-Ay +G(y)+f) -  \psi^{\mu_{N+1}}\sum_{k=1}^N y_k'.
\end{align*}

Using $\psi'$ in \eqref{dev} and estimate \eqref{vNrate}, we have 
\begin{align*}
 &\mu_{N+1}\psi(t)^{\mu_{N+1}-1}\psi'(t) v_N(t)\\
 &=
 \begin{cases}
\mu_{N+1} w_N(t) &\text{ for Type 1,}\\
\mu_{N+1}\psi(t)^{\mu_{N+1}-1} \bigo(\psi(t)^{-\mu_N-\delta_N})& \text{ for Type 2,}\\
\mu_{N+1}\psi(t)^{\mu_{N+1}-1} \bigo(\psi(t)^{-\lambda}) \bigo(\psi(t)^{-\mu_N-\delta_N}),\forall\lambda>0, & \text{ for Type 3.}
\end{cases}
\end{align*}

Note  that $\mu_{N+1}\le \mu_N+1$ for Type 2, thanks to Assumption \ref{muSas}\ref{Ac}.
For Type 3, take $\lambda>\mu_N+1-\mu_{N+1}$.
Then
\beqs
 \mu_{N+1}\psi(t)^{\mu_{N+1}-1}\psi'(t) v_N(t)
=\epsilon \mu_{N+1} w_N(t) +\bigo(\psi(t)^{-\delta_N}).
\eeqs

Also,
\beqs
\psi^{\mu_{N+1}}Ay = \psi^{\mu_{N+1}} A(\bar y_N +v_N)
=  \psi^{\mu_{N+1}} A \bar y_N  + Aw_N,
\eeqs
and
\begin{align*}
\psi(t)^{\mu_{N+1}}f(t)
&=\psi(t)^{\mu_{N+1}} \Big[\bar f_{N}(t)+f_{N+1}(t)+\bigo(\psi(t)^{-\mu_{N+1}-\varep_{N+1} }\Big] \\
&= \psi(t)^{\mu_{N+1}} \bar f_{N}(t)+\tilde p_{N+1}(t) +\bigo(\psi(t)^{-\varep_{N+1}}).
\end{align*}

Thus,
\beq\label{w1}
\begin{aligned}
w_N'&= -(A-\epsilon\mu_{N+1}I_n)w_N - \psi(t)^{\mu_{N+1}}  A \bar y_N 
+  \psi(t)^{\mu_{N+1}}G(y)+\tilde p_{N+1}(t)\\
&\quad  +  \psi(t)^{\mu_{N+1}} \bar f_{N}(t)
 -  \psi(t)^{\mu_{N+1}}\sum_{k=1}^N y_k'
+\bigo(\psi(t)^{-\min\{\delta_N,\varep_{N+1}\}}).
\end{aligned}
\eeq

Below, two terms $\psi(t)^{\mu_{N+1}}G(y(t))$ and $\psi(t)^{\mu_{N+1}}\sum_{k=1}^N y_k'(t)$ in \eqref{w1} will be further calculated.

\medskip
(b) We calculate $\psi(t)^{\mu_{N+1}}G(y(t))$. Letting $\delta=\mu_1/2$ in \eqref{u-first} yields
\beq \label{yr}
y(t)= \bigo( \psi(t)^{-\mu_1/2}).
\eeq

Let $M_{N+1}$ be the smallest integer such that 
\beq \label{MN}
M_{N+1}\ge \frac{2\mu_{N+1}}{\mu_1}.
\eeq 

Note that $ M_{N+1}\ge 2$. By \eqref{Grem1}, there exists  $\theta_N>0$ such that
 \beq\label{Grem3}
 \Big|G(y)-\sum_{m=2}^{M_{N+1}} G_m(y)\Big|=\bigo(|y|^{M_{N+1}+\theta_N})\text{ as } y\to 0.
 \eeq
 
We calculate and estimate, using \eqref{Grem3} and \eqref{yr},
\begin{align*}
\psi(t)^{\mu_{N+1}}G(y(t))
&= \psi(t)^{\mu_{N+1}} \Big\{\sum_{m=2}^{M_{N+1}} G_m(y(t)) + \bigo(|y(t)|^{M_{N+1}+\theta_N}) \Big\}\\
&= \psi(t)^{\mu_{N+1}} \sum_{m=2}^{M_{N+1}} G_m(y(t)) +\psi(t)^{\mu_{N+1}}  \bigo(\psi(t)^{-(M_{N+1}+\theta_N)\mu_1/2}).
\end{align*}
Thus,
\beq\label{preG}
\psi(t)^{\mu_{N+1}}G(y(t))
=\psi(t)^{\mu_{N+1}} \sum_{m=2}^{M_{N+1}} G_m(y(t)) +  \bigo\left(\psi(t)^{-\delta_{N+1}'}\right),
\eeq
where $\delta_{N+1}'=(M_{N+1}+\theta_N)\mu_1/2-\mu_{N+1}$, which is a positive number thanks to \eqref{MN}.

For each $G_m(y(t))$ in \eqref{preG}, we rewrite and estimate it, using \eqref{GG} and \eqref{multineq}, as 
\begin{align*}
G_m(y(t))&= G_m(\bar y_N+v_N) 
 =\mathcal G_m(\bar y_N+v_N,\ldots,\bar y_N+v_N)\\
 &=\mathcal G_m(\bar y_N,\ldots,\bar y_N)+\mathcal G_m(v_N,\ldots,v_N)+\sum_{k=1}^{m-1}\bigo(|\bar y_N(t)|^k|v_N(t)|^{m-k})\\
 &=G_m(\bar y_N(t)) +\bigo(|v_N(t)|^2) +\bigo(|\bar y_N(t)| |v_N(t)|).
\end{align*}

The last two terms are estimated, by using \eqref{vNrate} and \eqref{yNbrate} with $\delta=\delta_N/2$, by
\beqs
\bigo(\psi(t)^{-2(\mu_N+\delta_N)})+\bigo(\psi(t)^{-\mu_1+\delta_N/2}\psi(t)^{-\mu_N-\delta_N }).
\eeqs

Since $\mu_{N+1}\le \mu_N+\mu_1\le 2\mu_N$, we obtain
\beq\label{Gmb}
 G_m(y(t))=G_m(\bar y_N(t))+\bigo(\psi(t)^{-\mu_{N+1}-\delta_N/2}).
\eeq

Summing up \eqref{Gmb} in $m$ and combining with \eqref{preG}, we obtain
\beq\label{psG}
\psi(t)^{\mu_{N+1}}G(y(t))
= \psi(t)^{\mu_{N+1}}\sum_{m=2}^{M_{N+1}} G_m(\bar y_N(t))+\bigo(\psi(t)^{-\delta_N/2 })+  \bigo\left(\psi(t)^{-\delta_{N+1}'}\right).
\eeq

We continue to manipulate
\beq\label{sGm}
\sum_{m=2}^{M_{N+1}} G_m(\bar y_N(t))
=  \sum_{m=2}^{M_{N+1}}\sum_{1\le j_{1},j_{2},\ldots,j_{m}\le N} \frac{ \mathcal G_m( \tilde q_{j_{1}}(t), \tilde q_{j_{2}}(t),\ldots , \tilde q_{j_{m}}(t))}{\psi(t)^{\mu_{j_{1}}+\mu_{j_{2}}+\ldots \mu_{j_{m}}}}.
\eeq 

Note from \eqref{multineq}, the fact that each $q_{j_i}$ is a polynomial, and relation \eqref{ratiolim}, that
\beq\label{Gq}
|\mathcal G_m( \tilde q_{j_{1}}(t), \tilde q_{j_{2}}(t),\ldots , \tilde q_{j_{m}}(t))|
\le \|\mathcal G_m\| \cdot \prod_{i=1}^m |q_{j_i}(\phi(t))|=\bigo(\psi(t)^\delta)
\quad\forall \delta>0.
\eeq

Thanks to Assumption \ref{muSas}\ref{Aa}, the set $\mathcal S$ preserves the addition. Hence, the sum $\mu_{j_{1}}+\mu_{j_{2}}+\ldots \mu_{j_{m}}$ belongs to $\mathcal S$, and, thus, it must be $\mu_k$ for some $k\ge 1$.
Therefore, we can split the sum in \eqref{sGm} into three parts:
\beqs 
\mu_{j_{1}}+\mu_{j_{2}}+\ldots \mu_{j_{m}}=\mu_k\text{ for }
k\le N,\ 
k=N+1\text{ and }k\ge N+2.
\eeqs 

Corresponding to the last part, i.e., $\mu_k\ge \mu_{N+2}$, taking into account \eqref{Gq}, the summand in \eqref{sGm} is 
\beqs 
\frac{ \mathcal G_m( \tilde q_{j_{1}}(t), \tilde q_{j_{2}}(t),\ldots , \tilde q_{j_{m}}(t))}{\psi(t)^{\mu_{j_{1}}+\mu_{j_{2}}+\ldots \mu_{j_{m}}}}
=\bigo(\psi(t)^{-\mu_{N+2}+\delta })
\quad\forall \delta>0.
\eeqs 
Thus, we rewrite \eqref{sGm} as
\beq\label{sGm2}
\sum_{m=2}^{M_{N+1}} G_m(\bar y_N(t))
= \sum_{k=1}^{N+1} \frac{\tilde{Q}_k(t)}{\psi(t)^{\mu_k}}  +\bigo(\psi(t)^{-\mu_{N+2}+\delta })\quad\forall \delta>0,
\eeq 
where, for $1\le k\le N+1$,
\beq\label{nonmix}
\tilde{Q}_k(t)= \sum_{m=2}^{M_{N+1}}\sum_{\substack{{1\le j_{1},j_{2},\ldots,j_{m}\le N}\\ \mu_{j_{1}}+\mu_{j_{2}}+\ldots \mu_{j_{m}}=\mu_k}} \mathcal G_m( \tilde q_{j_{1}}(t), \tilde q_{j_{2}}(t),\ldots , \tilde q_{j_{m}}(t)) = Q_k(\phi(t)),
\eeq
with $Q_k:\R\to\R^n$ being defined by
\beq\label{Ik}
Q_k(z)= \sum_{m=2}^{M_{N+1}}\sum_{\substack{{1\le j_{1},j_{2},\ldots,j_{m}\le N}\\ \mu_{j_{1}}+\mu_{j_{2}}+\ldots \mu_{j_{m}}=\mu_k}} \mathcal G_m(q_{j_{1}}(z),q_{j_{2}}(z),\ldots ,q_{j_{m}}(z))\text{ for $z\in\R$.}
\eeq

For $1\le k\le N+1$, we note that $N\ge k-1$, and, thanks to \eqref{MN}, $M_{N+1}>\mu_{N+1}/\mu_1>\mu_k/\mu_1$.
By relation \eqref{sumequiv},
\beq\label{cls}
\sum_{m=2}^{M_{N+1}}\sum_{\substack{{1\le j_{1},j_{2},\ldots,j_{m}\le N}\\ \mu_{j_{1}}+\mu_{j_{2}}+\ldots \mu_{j_{m}}=\mu_k}}
=\sum_{m\ge 2}\sum_{\mu_{j_{1}}+\mu_{j_{2}}+\ldots \mu_{j_{m}}=\mu_k}.
\eeq
Thus,  we have
\beq\label{Ik2}
Q_k= \sum_{m\ge2}\sum_{\mu_{j_{1}}+\mu_{j_{2}}+\ldots \mu_{j_{m}}=\mu_k} \mathcal G_m(q_{j_{1}},q_{j_{2}},\ldots ,q_{j_{m}}).
\eeq

Combining \eqref{psG} with \eqref{sGm2} gives, for any $\delta>0$,
\beq\label{wGy}
\begin{aligned}
\psi(t)^{\mu_{N+1}}G(y(t))
&=\tilde{Q}_{N+1}(t) +\psi(t)^{\mu_{N+1}} \sum_{k=1}^N \frac{\tilde{Q}_k(t)}{\psi(t)^{\mu_k}} \\
&\quad  +\bigo(\psi(t)^{-\mu_{N+2}+\mu_{N+1}+\delta }) +\bigo(\psi(t)^{-\delta_N/2 }) +  \bigo\left(\psi(t)^{-\delta_{N+1}'}\right).
\end{aligned}
 \eeq

Combining \eqref{w1} with \eqref{wGy} gives
\beq\label{w100}
\begin{aligned}
w_N'&= -(A-\epsilon\mu_{N+1}I_n)w_N + \tilde{Q}_{N+1} +\tilde p_{N+1}\\
&\quad + \psi(t)^{\mu_{N+1}} \sum_{k=1}^N  \psi(t)^{-\mu_k} \Big\{-A\tilde q_k + \tilde{Q}_k+\tilde p_k \Big\}\\
&\quad -  \psi(t)^{\mu_{N+1}}\sum_{k=1}^N y_k'
+\bigo(\psi(t)^{-\varep'_{N+1}}) + \bigo(\psi(t)^{-\mu_{N+2}+\mu_{N+1}+\delta}),
\end{aligned}
\eeq
for any $\delta>0$, where $\varep'_{N+1}=\min\{\delta_N/2,\varep_{N+1},\delta_{N+1}'\}>0$.

(c) We calculate $\psi(t)^{\mu_{N+1}}\sum_{k=1}^N y_k'$. Define $\tilde \chi_k(t)= \chi_k(\phi(t))$ for $1\le k\le N+1$, and let
\beqs 
\theta=
 \begin{cases}
  1, \quad \text{ for Type 2,}\\
 0, \quad \text{ for Types 1 and 3.}
 \end{cases}
\eeqs

Let $\lambda>0$ be arbitrary. Using \eqref{dev}, we have
\begin{align*}
 y_k'(t)
 &= q_k'(\phi(t))\phi'(t) \psi(t)^{-\mu_k}-\mu_k q_k(\phi(t))\psi(t)^{-\mu_k-1} \psi'(t)\\
 &= \epsilon [q_k'(\phi(t))-\mu_k q_k(\phi(t))] \psi(t)^{-\mu_k}
 + \theta [q_k'(\phi(t)) -\mu_k q_k(\phi(t))]\psi(t)^{-\mu_k-1}\\
 &\quad + (1-\epsilon)(1-\theta)\bigo(\psi(t)^{-\lambda}).
\end{align*}

Summing up in $k$ from $1$ to $N$ gives
\beq\label{syp}
\sum_{k=1}^N y_k'(t)
 =  \sum_{k=1}^N   \epsilon [q_k'(\phi(t))-\mu_k q_k(\phi(t))]\psi(t)^{-\mu_k} -\theta J(t) + \bigo(\psi(t)^{-\lambda}),
\eeq
where $J(t)=\sum_{k=1}^N [\mu_k q_k(\phi(t))-q_k'(\phi(t)) ]\psi(t)^{-\mu_k-1} $.

Consider $J(t)$ when $\theta=1$, i.e., in the case of Type 2.
Note, by Assumption \ref{muSas}\ref{Ac}, that $\mu_k+1=\mu_s$ for a unique $s\in\N$, and, in case $1\le s\le N+1$,
\beqs
[\mu_k q_k(\phi(t))-q_k'(\phi(t)) ]\psi(t)^{-\mu_k-1}= \chi_s(\phi(t))\psi(t)^{-\mu_s}.
\eeqs

Splitting the sum in $J(t)$ into $s\le N$, $s= N+1$ and $s\ge N+2$, we obtain
\beq\label{tJ}
\theta J(t) =  \sum_{k=1}^N \tilde \chi_k(t)  \psi(t)^{-\mu_k } +  \tilde \chi_{N+1}(t) \psi(t)^{-\mu_{N+1}}
  +  \theta \sum_{\substack{1 \le k \le N \\ \mu_k+1 \ge \mu_{N+2}}} \frac{\mu_k q_k(\phi(t))-q'_k(\phi(t))}{\psi(t)^{\mu_k+1}}.
\eeq 

Note, for $\mu_k+1\ge \mu_{N+2}$, we have
\beq\label{smm}
 \frac{\mu_k q_k(\phi(t))-q'_k(\phi(t))}{\psi(t)^{\mu_k+1}}=\bigo(\psi(t)^{-\mu_{N+2}+\delta}) \quad \forall \delta>0.
\eeq

For any $\delta>0$, selecting $\lambda>\mu_{N+2}$ in \eqref{syp}, and using \eqref{tJ}, \eqref{smm}, we obtain 
\begin{align*}
\sum_{k=1}^N y_k'(t)
 &=\sum_{k=1}^N\Big\{ \epsilon [q_k'(\phi(t))-\mu_k q_k(\phi(t))] 
- \tilde \chi_k(t)\Big\}\psi(t)^{-\mu_k} \\
 &\quad -   \tilde \chi_{N+1}(t) \psi(t)^{-\mu_{N+1}} + \bigo(\psi(t)^{-\mu_{N+2}+\delta}).
\end{align*}
and, hence,
\beq\label{ykprime}
\begin{split}
\psi(t)^{\mu_{N+1}}\sum_{k=1}^N y_k'(t)
&=\psi(t)^{\mu_{N+1}}\sum_{k=1}^N\Big\{ \epsilon [q_k'(\phi(t))-\mu_k q_k(\phi(t))] - \tilde \chi_k(t)\Big\}\psi(t)^{-\mu_k}\\
&\quad - \tilde \chi_{N+1}(t)  + \bigo(\psi(t)^{-\mu_{N+2}+\mu_{N+1}+\delta}).
\end{split}
\eeq

Combing \eqref{w100} with \eqref{ykprime}, and selecting $\delta=(\mu_{N+2}-\mu_{N+1})/2$ give
\beq \label{wN2}
\begin{aligned}
w_N'&=-(A-\epsilon\mu_{N+1} I_n)w_N + \hat\Phi_{N+1}(\phi(t)) +\psi(t)^{\mu_{N+1}} \sum_{k=1}^N \psi(t)^{-\mu_k}\tilde \Phi_k(t)\\
&\quad  +  \bigo(\psi(t)^{-\widehat\delta_{N+1}}),
\end{aligned}
\eeq
where $\widehat\delta_{N+1}=\min\{\varep'_{N+1},(\mu_{N+2}-\mu_{N+1})/2\}>0$, $\hat\Phi_{N+1}:\R\to \R^n$ is the polynomial defined by
\beqs
\hat\Phi_{N+1}=Q_{N+1}+p_{N+1} + \chi_{N+1},
\eeqs
and, for $1\le k\le N$,
\begin{align*}
\tilde \Phi_k(t)&=\Phi_k(\phi(t)),\text{ with }
\Phi_k=- Aq_k + Q_k + p_k -\epsilon (q_k'-\mu_k q_k) + \chi_k.
\end{align*}

By the induction hypothesis and \eqref{Ik2}, we have $\Phi_k=0$ for $k=1,\ldots,N$.
Thus,
\beq\label{wN}
w_N'=-(A-\epsilon\mu_{N+1} I_n)w_N +\hat\Phi_{N+1}(\phi(t)) + \bigo(\psi(t)^{-\widehat\delta_{N+1}}).
\eeq

(d) Using equation \eqref{wN}, we apply Lemma \ref{exp-odelem} in case of Type 1, or Corollary \ref{powlog-odelem} in case of Types 2 and 3, to polynomial 
$p=\hat\Phi_{N+1}$.

We check condition \eqref{ellams} for Type 1 with $\lambda=\mu_{N+1}$. Let $\lambda_*$ be as in \eqref{ellams}. 
Then $\lambda_*<\mu_{N+1}$. Also, since $\lambda_*$ is an eigenvalue of $A$, we have, thanks to Assumption \ref{muSas}\ref{Ab}, $\lambda_*\in\mathcal S$.  Hence $\lambda_*\le \mu_N$. Thus,
\beqs
e^{(\lambda_*-\mu_{N+1})t}|w_N(t)|=e^{\lambda_* t}|v_N(t)|\le e^{\mu_N t}|v_N(t)|\to 0\text{ as }t\to\infty.
\eeqs
Hence \eqref{ellams} is satisfied.

Then, by Lemma \ref{exp-odelem} and Corollary \ref{powlog-odelem}, there exist a polynomial $q_{N+1}:\R\to\R^n$ and a number $\delta_{N+1}>0$ such that
\beq \label{wNplus1}
|w_N(t)-q_{N+1}(\phi(t))|=\bigo(\psi(t)^{-\delta_{N+1}}),
\eeq
and $q_{N+1}(z)$ solves
\beqs
Aq_{N+1}(z) + \epsilon [q_{N+1}'(z)-\mu_{N+1} q_{N+1}(z)] =\hat\Phi_{N+1}(z),\quad z\in\R.
\eeqs

Multiplying \eqref{wNplus1} by $\psi(t)^{-\mu_{N+1} } $ gives
\beqs
|y(t)-\sum_{k=1}^{N+1} q_k(\phi(t)) \psi(t)^{-\mu_k} |=\bigo(\psi(t)^{-\mu_{N+1} - \delta_{N+1}}).
\eeqs

Thus, the statement ($\mathcal T_{N+1}$) holds true.

\medskip
\textbf{Conclusion.} By the induction principle, the statement  ($\mathcal T_N$) is true for all $N\in\N$. 
Note that the polynomial $q_{N+1}$ is constructed without changing the previous $q_k$ for $1\le k \le N$. Therefore, the polynomials $q_k$ exist for all $k\in\N$, for which  ($\mathcal T_N$) is true for all $N\in\N$.
Consequently, we obtain the asymptotic expansion \eqref{ymu} with the polynomials $q_k$'s satisfying \eqref{pq1} and \eqref{pq-k}.
\end{proof}

For convenience in comparisons, we write formulas in Theorem \ref{mainthm} for three cases of $(\psi,\phi)$ explicitly.

\medskip
\noindent\textbf{Type 1.} The expansions \eqref{fmu} and \eqref{ymu} are  
\beqs
 f(t)\sim \sum_{k=1}^\infty p_k(t)e^{-\mu_k t}=\sum_{k=1}^\infty f_k(t)\text{ and } 
 y(t)\sim \sum_{k=1}^\infty q_k(t)e^{-\mu_k t}=\sum_{k=1}^\infty y_k(t),
 \eeqs
 where, following the concise form \eqref{pqall},
\beq \label{Tqk}
q_k'+(A-\mu_k I_n)q_k=\sum_{m\ge 2} \sum_{\mu_{j_{1}}+\mu_{j_{2}}+\ldots \mu_{j_{m}}=\mu_k} \mathcal G_m(  q_{j_{1}}, q_{j_{2}},\ldots ,  q_{j_{m}}) +p_k  \quad\text{for } k \in\N,
\eeq 
or, equivalently,
 \beqs
y_k'+Ay_k=\sum_{m\ge 2} \sum_{\mu_{j_{1}}+\mu_{j_{2}}+\ldots \mu_{j_{m}}=\mu_k} \mathcal G_m(  y_{j_{1}}, y_{j_{2}},\ldots ,  y_{j_{m}}) +f_k  \quad\text{for } k \in\N.
\eeqs 
\medskip
\noindent\textbf{Type 2.} The expansions \eqref{fmu} and \eqref{ymu} are  
 \beqs
 f(t)\sim \sum_{k=1}^\infty p_k(\ln t)t^{-\mu_k}\text{ and }
 y(t)\sim \sum_{k=1}^\infty q_k(\ln t)t^{-\mu_k},
 \eeqs
where 
 \beq\label{qtype2}
q_k=A^{-1}\Big\{\sum_{m\ge 2} \sum_{\mu_{j_{1}}+\mu_{j_{2}}+\ldots \mu_{j_{m}}=\mu_k} \mathcal G_m(  q_{j_{1}}, q_{j_{2}},\ldots ,  q_{j_{m}}) +p_k + \chi_k\Big\}\text{ for }k\in\N,
 \eeq
with $\chi_k$ defined by \eqref{chi1} and \eqref{chik}.

If $\mu_k=k$ for all $k\in\N$, then
$
\chi_k=(k-1)q_{k-1}-q_{k-1}'
$
for all $k\ge 2$.

 In particular, if $p_k= \eta_k=const.$ for all $k$, then $q_k=\xi_k=const.$ for all $k$. This type of expansions is studied  in \cite{CaH1} for the Navier--Stokes equations.

 \medskip
\noindent \textbf{Type 3.} The expansions \eqref{fmu} and \eqref{ymu} are  
 \beqs
 f(t)\sim \sum_{k=1}^\infty p_k(\iln_{n_*}(t))\iln_{m_*}(t)^{-\mu_k}\text{ and }
 y(t)\sim \sum_{k=1}^\infty q_k(\iln_{n_*}(t))\iln_{m_*}(t)^{-\mu_k},
 \eeqs
where 
 \beq\label{qtype3}
q_k=A^{-1}\Big\{\sum_{m\ge 2} \sum_{\mu_{j_{1}}+\mu_{j_{2}}+\ldots \mu_{j_{m}}=\mu_k} \mathcal G_m(  q_{j_{1}}, q_{j_{2}},\ldots ,  q_{j_{m}}) +p_k\Big\} \text{ for }k\in\N.
 \eeq

\begin{remark}
We have the following comparisons.
\begin{enumerate}[label={\rnum}]
 \item  In case of Type 1, the asymptotic expansion, in general, depends on the individual solution $y(t)$, and $q_k$ is determined by solving an ODE.
In cases of Types 2 and 3, all solutions have the same expansion, and $q_k$ is determined by some algebraic operations.
 
 \item For Types 1 and 2, the time derivative in  \eqref{sys-eq} is reflected on the construction of $q_k$'s, see \eqref{Tqk} and \eqref{chik}, respectively. This is not the case for Type 3, see \eqref{qtype3}.
 
 \item In case of Type 1 and when $\mu_k$ is not an eigenvalue of $A$, thanks to Remark \ref{rmkoq}\ref{fsb}, the function $q_k(t)$ is the unique polynomial solution of \eqref{Tqk}, which depends only on the forcing function $f(t)$ and the previous $q_j(t)$, for $1\le j\le k$. Consequently, if two solutions of \eqref{sys-eq} have the same $q_k$'s for $1\le k\le N$ in their asymptotic expansions, where $N\in\N$ such that $\mu_N=\lambda_d$,
then they have the same $q_k$'s for all $k>N$, and, hence, for all $k\in\N$. This property is not available in infinite dimensional spaces.
\end{enumerate}
\end{remark}

\section{Expansions with multiple secondary base functions}\label{multisec}

This section aims to generalize the results in section \ref{basetwo}.
Subsection \ref{logsec} will generalize expansions for Type 3, while subsection \ref{doublexp} for Type 2.
With the notation in Definitions \ref{ELdef} and \ref{Pdef}, we can consider a general form of expansions in the following. 

\begin{definition}\label{mixdef}
Let $(X,\|\cdot\|_X)$ be a normed space, and $g$ be a function from $(T,\infty)$ to $X$ for some $T\in\R$. 
Let $m_*\in \Z_+$. 

\begin{enumerate}[label={\tnum}]
 \item Let $(\gamma_k)_{k=1}^\infty$ be a divergent, strictly increasing sequence of positive numbers, and $n_k\in\N$,  $p_k\in \classP(n_k,X)$ for each $k\in\N$. 
We say
\beq\label{fiter}
g(t)\sim \sum_{k=1}^\infty p_k(\LL_{m_*, n_k}(t))\iln_{m_*}(t)^{-\gamma_k},
\eeq
if, for each $N\in\N$, there is some $\mu>\gamma_N$ such that
\beq\label{iterap}
\Big\|g(t) - \sum_{k=1}^N p_k(\LL_{m_*,n_k}(t))\iln_{m_*}(t)^{-\gamma_k}\Big\|_X=\bigo(\iln_{m_*}(t)^{-\mu}).
\eeq

\item Let $N\in\N$, $(\gamma_k)_{k=1}^N$ be positive and strictly increasing, and $n_k\in\N$,  $p_k\in \classP(n_k,X)$, for $k=1,2,\ldots,N$. 
We say
\beqs
g(t)\sim \sum_{k=1}^N p_k(\LL_{m_*, n_k}(t))\iln_{m_*}(t)^{-\gamma_k},
\eeqs 
if it holds for all $\lambda>0$ that
\beqs 
\Big\|g(t) - \sum_{k=1}^N p_k(\LL_{m_*,n_k}(t))\iln_{m_*}(t)^{-\gamma_k}\Big\|_X=\bigo(\iln_{m_*}(t)^{-\lambda}).
\eeqs
\end{enumerate}
\end{definition}

In particular, when $m_*=0$, expansion \eqref{fiter} reads as
\beqs
g(t)\sim \sum_{k=1}^\infty p_k(\LL_{n_k}(t)) t^{-\gamma_k}.
\eeqs

We have the following remarks on Definition \ref{mixdef}.
\begin{enumerate}[label={\rnum}]
 \item Similar to Definition \ref{psi-phi}, we call $\iln_{m_*}(t)$ the \emph{primary base function} of expansion \eqref{fiter}, and $\iln_{m_*+j}(t)$ with $1\le j\le n_k$ and $k\in\N$, the \emph{secondary base functions}.

 \item Comparing two expansions  \eqref{expan} and \eqref{fiter} when they have the same primary base function, the latter is more general than the former, even in the case \eqref{fiter} has only one secondary base function.
 It is due to the fact that the functions $p_k$'s belong to a larger class, see remark \ref{Ca} after Definition \ref{Pdef}. 
 
 \item Because function $p_k(\LL_{m_*,n_k}(t))$ is not restricted to only non-negative integer powers of $\iln_{m_*+j}(t)$'s,  the asymptotic expansion \eqref{fiter}, in fact, is a more precise approximation of $g(t)$ compared to \eqref{expan} for Type 3 in Definition \ref{typedef}.
 \end{enumerate}

Note that
\beq\label{LLo}
\lim_{t\to\infty} \frac{\LL_{m_*, k}(t)^\alpha}{\iln_{m_*}(t)^\delta}=0 \text{ for any $k\in\N$, $\alpha\in \R^k$, and $\delta>0$.}
\eeq

Similar to Lemma \ref{ue}, we have the following uniqueness of the approximation \eqref{iterap}.

\begin{lemma}\label{ue2}
Let $(X,\|\cdot\|_X)$ be a normed space. 
Given a function $g:(T,\infty)\to X$ for some $T\in\R$. Let $N\in\N$, numbers $n_k\in\N$ and $\gamma_k\in\R$, for $1\le k\le N$, such that 
$0\le \gamma_1<\gamma_2<\ldots<\gamma_N$.
 Suppose there exist $p_k\in \classP(n_k,X)$, for $1\le k\le N$, such that  \eqref{iterap} holds for some $\mu>\gamma_N$.
Then such functions $p_1,p_2,\ldots,p_N$ are unique.
\end{lemma}
\begin{proof}
Suppose $\hat p_k\in \classP(n_k,X)$, for $1\le k\le N$, satisfy 
\beqs
\Big\|g(t) - \sum_{k=1}^N \hat p_k(\LL_{m_*,n_k}(t))\iln_{m_*}(t)^{-\gamma_k}\Big\|_X=\bigo(\iln_{m_*}(t)^{-\hat \mu})\text{ for some }\hat \mu>\gamma_N.
\eeqs

Let $h_k=p_k-\hat p_k$ for $1\le k\le N$. By triangle inequality, see \eqref{tri}, we have
\beqs
\Big\|\sum_{k=1}^N h_k(\LL_{m_*,n_k}(t))\iln_{m_*}(t)^{-\gamma_k}\Big\|_X=\bigo(\iln_{m_*}(t)^{-\bar \mu})\text{ for }\bar\mu=\min\{\mu,\hat\mu\}>\gamma_N.
\eeqs

Multiplying this equation by $\iln_{m_*}(t)^{\gamma_1}$ yields
\beq\label{h1L}
\|h_1(\LL_{m_*,n_k}(t))\|_X=\bigo(\iln_{m_*}(t)^{-\delta}) \text{ for } \delta=\bar\mu-\gamma_1>0.
\eeq

Suppose $h_1\ne 0$. Write $h_1(z)$ as a finite sum $\sum c_\beta z^\beta$ for $z\in (0,\infty)^{n_k}$, where $c_\beta$'s are non-zero vectors in $X$, and $\beta$'s are distinct powers in $\R^{n_k}$. 
We use the lexicography order for the powers $\beta$'s in $\R^{n_k}$.
If $\alpha,\beta$ are the powers in $\R^{n_k}$, and $\alpha>\beta$, then
 \beq\label{LLab}
 \lim_{t\to\infty}\frac{\LL_{m_*,n_k}(t)^\beta}{\LL_{m_*,n_k}(t)^\alpha}= 0.
 \eeq
Let $\beta_*$ be the maximum power among those $\beta$'s. 
Then multiplying \eqref{h1L} by $(\LL_{m_*,n_k}(t))^{-\beta_*}$ and passing $t\to\infty$, making use of \eqref{LLo} and \eqref{LLab}, we obtain $c_{\beta_*}=0$, which is a contradiction. Thus, we have $h_1=0$.
Repeating this argument gives $h_k=0$, hence, $p_k=\hat p_k$, for all $k=1,2,\ldots, N$. 
\end{proof}

\medskip
\emph{Throughout this section, $f(t)$ is a forcing function as in Assumption \ref{assumpf}, and $y(t)$ is a solution of \eqref{sys-eq} as in Assumption \ref{assumpu}.}

\subsection{Iterated logarithmic expansions}\label{logsec}

This subsection studies the expansions that contain only iterated logarithmic functions.

\begin{assumption} \label{fLas}
The function $f(t)$ admits the asymptotic expansion, in the sense of Definition \ref{mixdef} with $X=\R^n$ and $\|\cdot\|_X=|\cdot|$, 
\beq \label{fLL}
f(t)\sim \sum_{k=1}^\infty p_k(\LL_{m_*,n_k}(t))\iln_{m_*}(t)^{-\mu_k},
\eeq
where $m_*\in \N$, $(\mu_k)_{k=1}^\infty$ is a divergent, strictly increasing sequence of positive numbers, the set $\mathcal S:=\{\mu_k:k\in\N\}$  preserves the addition, $n_k$ is increasing in $k$, but not necessarily strictly increasing, and $p_k\in \classP(n_k,\R^n)$.
\end{assumption}

Below are two typical cases for Assumption \ref{fLas} to hold.

\begin{scenario}\label{scen4}
Let $m_*\in \N$. Assume, in the sense of Definition \ref{mixdef}, that
\beq\label{fLraw}
f(t)\sim \sum_{k=1}^\infty \hat p_k(\LL_{m_*, \hat n_k}(t))\iln_{m_*}(t)^{-\alpha_k},
\eeq
where $\hat n_k\in\N$, $\hat p_k\in\classP(\hat n_k,\R^n)$, and $(\alpha_k)_{k=1}^\infty$ is a divergent, strictly increasing sequence of positive numbers. 

Let the set $\mathcal S$ be defined by \eqref{S3} and be re-ordered as the sequence $(\mu_k)_{k=1}^\infty$ as in \eqref{Sseq}.
Then $(\alpha_k)_{k=1}^\infty$ becomes a subsequence of $(\mu_k)_{k=1}^\infty$, and, by re-indexing $\hat p_k$ and $\hat n_k$, we rewrite \eqref{fLraw} as
\beq\label{fLraw1}
f(t)\sim \sum_{k=1}^\infty \hat p_k(\LL_{m_*, \hat n_k}(t))\iln_{m_*}(t)^{-\mu_k}.
\eeq

Let $n_k=\max\{\hat n_j:1\le j\le k\}$. By embedding \eqref{Pembed}, we have $\classP(\hat n_k,\R^n)\subset \classP(n_k,\R^n)$ and define $p_k=\mathcal I_{\hat n_k,n_k}\hat p_k$.
Then $p_k\in \classP(n_k,\R^n)$, and, thanks to \eqref{Ikm},  $\hat p_k(\LL_{m_*, \hat n_k}(t))=p_k(\LL_{m_*,n_k}(t))$.
Thus, we can rewrite \eqref{fLraw1} as \eqref{fLL}.
\end{scenario}

\begin{scenario}\label{scen5}
Assume, similar to \eqref{fLraw}, we have the finite expansion, in the sense of Definition \ref{mixdef},
\beq\label{fLraw2}
f(t)\sim \sum_{k=1}^N \hat p_k(\LL_{m_*, \hat n_k}(t))\iln_{m_*}(t)^{-\alpha_k},
\eeq
for some $N\in \N$. Let $\mathcal S$ be defined by \eqref{S3} for $1\le \ell_j\le N$. Then similar to Scenario \ref{scen4}, we can rewrite \eqref{fLraw2} as \eqref{fLL}.
\end{scenario}

Our second main result on the asymptotic expansion for solutions of \eqref{sys-eq} is the following.

\begin{theorem}\label{logthm}
Under Assumption \ref{fLas}, the solution $y(t)$ admits the  asymptotic expansion, in the sense of Definition \ref{mixdef}, 
  \beq \label{yLL}
  y(t)\sim \sum_{k=1}^\infty q_k(\LL_{m_*,n_k}(t)) \iln_{m_*} (t)^{-\mu_k },
  \eeq 
where 
\beq\label{qinP} 
q_k\in\classP(n_k,\R^n)\text{ for all $k\in\N$,}
\eeq 
and are defined recursively by
\beq\label{log-q}
  q_k=
  \begin{cases}
A^{-1}p_1,&\text{ for $k=1$,}\\
A^{-1}\Big(\begin{displaystyle}
\sum_{m\ge 2}\ \sum_{\substack{ j_{1},j_{2},\ldots, j_{m} \ge 1 \\ \mu_{j_{1}}+\mu_{j_{2}}+\ldots + \mu_{j_{m}}=\mu_k}}
\end{displaystyle} 
\mathcal G_m(q_{j_{1}},q_{j_{2}},\ldots , q_{j_{m}}) +p_k\Big),
  &\text{ for  $k\ge 2$.}
  \end{cases}
\eeq
 \end{theorem}

 We quickly verify, by induction, that the definition of $q_k$ is valid and \eqref{qinP} holds true.
 
 First, because $p_1\in \classP(n_1,\R^n)$ and $q_1=A^{-1}p_1$, we have $q_1\in\classP(n_1,\R^n)$, thanks to Lemma \ref{Pprop}\ref{P2}.  
 
 Let $k\ge 2$. Suppose $q_j\in\classP(n_j,\R^n)$ for all $1\le j\le k-1$. 
 By Remark \ref{Tc} after the statement of Theorem \ref{mainthm}, the sums on the right-hand side of \eqref{log-q} is over finitely many $m$'s and $1\le j_1,j_2,\ldots,j_k\le k-1$. Thus, $n_{j_i}<n_k$, for $1\le i\le k$,  and, thanks to embedding \eqref{Pembed}, we can consider each $q_{j_i}$ belonging to $\classP(n_k,\R^n)$. Thus, by Lemma \ref{Pprop}\ref{P1}, $\mathcal G_m(q_{j_{1}},q_{j_{2}},\ldots , q_{j_{m}})$ is in $\classP(n_k,\R^n)$. Together with $p_k\in\classP(n_k,\R^n)$ and, again,   Lemma \ref{Pprop}\ref{P2}, we obtain $q_k\in\classP(n_k,\R^n)$.
 
 Then by the induction principle, $q_k\in\classP(n_k,\R^n)$ for all $k\in\N$.
 
 \begin{proof}[Proof of Theorem \ref{logthm}]
The proof is similar to that of Theorem \ref{mainthm} for Type 3. We sketch it here. 

Replace the statement ($\mathcal T_N$) in the proof of Theorem \ref{mainthm} by the following:
\beqs
\Big| y(t)- \sum_{k=1}^N q_k(\LL_{m_*,n_k}(t)) \iln_{m_*} (t)^{-\mu_k }\Big|=\bigo( \iln_{m_*} (t)^{-\mu_N-\delta_N})
\text{ for some $\delta_N>0$. }
\eeqs

Let $\psi(t) = \iln_{m_*}(t)$,  $\tilde p_k(t) = p_k(\LL_{m_*,n_k}(t))$, and define $f_k(t)$, $\bar f_N(t)$ as in \eqref{fkdef}.  

First, we notice that 
\beqs
 |\psi'(t)| = \bigo (\psi(t)^{-\lambda})\quad \forall \lambda >0.
\eeqs

One can verify \eqref{ffirst} and, hence,  \eqref{u-first}.

We prove that ($\mathcal T_N$) is true for all $N\in\N$.
We proceed by induction. Consider $t$ sufficiently large in all calculations below.

\textbf{First step: $N=1.$} Set $w_0(t)=\psi(t)^{\mu_1}y(t)$. We follow the calculations in the proof of Theorem \ref{mainthm} for Type 3 with $\epsilon=\theta=0$. We obtain the same equation \eqref{woeq} for $w_0$.
This gives
 \begin{align*}
 w_0' &=-A w_0 + p_1(\LL_{m_*,n_1}(t)) + \bigo(\psi(t)^{-\widehat\delta_1})\text{ for some }\widehat\delta_1>0.
 \end{align*}

  Applying Lemma \ref{iterlog-odelem} to this equation for $w_0$ yields the existence of a  number $\delta_1>0$ such that
\beqs
|w_0(t)-A^{-1}p_1(\LL_{m_*,n_1}(t))|= \bigo(\psi(t)^{-\delta_1 }).
\eeqs

Noting that $A^{-1}p_1=q_1 \in \classP(n_1,\R^n)$, and multiplying the preceding equation by $\psi(t)^{-\mu_1}$ gives
 \beqs
|y(t)-q_{1}(\LL_{m_*,n_1}(t)) \psi(t)^{-\mu_1 } |=\bigo(\psi(t)^{-\mu_1 - \delta_1 }).
\eeqs

Then ($\mathcal T_1$) holds true.

\textbf{ Induction step: $N\ge 1$.} Assume ($\mathcal T_N$).
Let $\tilde q_k(t)=q_k(\LL_{m_*,n_k}(t))$, define $y_k(t)$, $\bar y_N(t)$ and $v_N(t)$ as in \eqref{ys}.

Let $w_N(t)=\psi(t)^{\mu_{N+1}}v_N(t)$. 
Same calculations as in parts (a) and (b) of the proof of Theorem \ref{mainthm}, we obtain  equation \eqref{w100} again, which yields
\beq\label{ww2}
\begin{aligned}
w_N'&= -A w_N + \tilde{Q}_{N+1} +\tilde p_{N+1} + \psi(t)^{\mu_{N+1}} \sum_{k=1}^N  \psi(t)^{-\mu_k} \Big\{-A\tilde q_k + \tilde{Q}_k+\tilde p_k \Big\}\\
&\quad -  \psi(t)^{\mu_{N+1}}\sum_{k=1}^N y_k'
+\bigo(\psi(t)^{-\varep''_{N+1}}),
\end{aligned}
\eeq
for some $\varep''_{N+1}>0$. Here, same as \eqref{nonmix}, \eqref{Ik}, \eqref{cls}, and \eqref{Ik2}, for $1\le k\le N+1$,
\beq\label{QQ}
\tilde Q_k(t)=Q_k(\LL_{m_*,n_k}(t)), \text{ with $Q_k$ being defined by \eqref{Ik2}.}
\eeq

Note that $\psi'(t),\iln_k'(t)=\bigo(t^{-\lambda})$ for all $\lambda\in(0,1)$, and, by \eqref{qinP} and Lemma \ref{Pprop}\ref{P6}, 
$\partial q_k/\partial z_j \in\classP(n_k,\R^n)$ for $1\le j\le n_k$. Hence
\beqs
\ddt q_k(\LL_{m_*,n_k}(t))=\sum_{j=1}^{n_k} \frac{\partial q_k}{\partial z_j}(\LL_{m_*,n_k}(t)) \iln_{m_*+j}'(t)=\bigo(t^{-\lambda})\quad\forall \lambda\in(0,1).
\eeqs
Thus,
\beqs
 y_k'(t)
= \ddt q_k(\LL_{m_*,n_k}(t)) \psi(t)^{-\mu_k}-\mu_k q_k(\LL_{m_*,n_k}(t))\psi(t)^{-\mu_k-1} \psi'(t) \\
=\bigo(t^{-\lambda}) \quad \forall \lambda \in(0,1).
\eeqs 
Consequently, $ y_k'(t)=\bigo(\psi(t)^{-\mu_{N+1}-\varep''_{N+1}})$.
From this and \eqref{ww2}, we obtain \eqref{wN2} again:
\beqs
w_N'=-Aw_N + \hat\Phi_{N+1}(\LL_{m_*,n_{N+1}}(t)) +\psi(t)^{\mu_{N+1}} \sum_{k=1}^N \psi(t)^{-\mu_k}\Phi_k(\LL_{m_*,n_k}(t))  +  \bigo(\psi(t)^{-\widehat\delta_{N+1}}),
\eeqs
where $\widehat\delta_{N+1}=\varep''_{N+1}>0$,  and 
$\hat\Phi_{N+1}=Q_{N+1} +p_{N+1},$ with 
$\Phi_k =-Aq_k + Q_k +  p_k$ for $1\le k\le N$.

By \eqref{log-q}, $\Phi_k=0$ for $1\le k\le N$.
We therefore obtain
\beq\label{wN3}
w_N'=-Aw_N +\hat\Phi_{N+1}(\LL_{m_*,n_{N+1}}(t)) + \bigo(\psi(t)^{-\widehat\delta_{N+1}}).
\eeq

Using relation in \eqref{Pembed}, we can view $\hat\Phi_{N+1}$ as an element in the class $\classP(n_{N+1},\R^n)$.
By applying Lemma \ref{iterlog-odelem} to equation \eqref{wN3} with $y=w_N$ and $p=\hat\Phi_{N+1}$, there exists $\delta_{N+1}>0$  such that
\beq \label{ilog-wN}
|w_N(t)-A^{-1}\hat\Phi_{N+1}(\LL_{m_*,n_{N+1}}(t))|=\bigo(\psi(t)^{-\delta_{N+1}}).
\eeq

Note that $A^{-1}\hat\Phi_{N+1}=q_{N+1}$ which belongs to $\classP(n_{N+1},\R^n)$, thanks to Lemma \ref{Pprop}\ref{P2}. Multiplying equation \eqref{ilog-wN} by  $\psi(t)^{-\mu_{N+1}}$ gives
\beqs
|v_N(t)-q_{N+1}(\LL_{m_*,n_{N+1}}(t)) \psi(t)^{-\mu_{N+1} } |=\bigo(\psi(t)^{-\mu_{N+1} - \delta_{N+1} }).
\eeqs

This proves ($\mathcal T_{N+1}$). Therefore, by the induction principle, ($\mathcal T_N$) holds for all $N\in\N$, which implies the asymptotic expansion \eqref{yLL}.
\end{proof}

\begin{remark}
 If all functions $p_k$'s in \eqref{fLL} are polynomials, then, by induction, all $q_k$'s in \eqref{yLL} are also polynomials.
\end{remark}

\subsection{Mixed power and iterated logarithmic expansions}\label{doublexp}

For motivation, we consider a simple case when 
$f(t)=t^{-1}\ln \ln t$ for large $t$, with $\ln\ln t$ denoting $\ln(\ln t)$.
Then the expansion of solution $y(t)$ should contain at least a term $t^{-1}\ln \ln t$.
It yields that $y'(t)$ should contain 
$$ \frac{1}{t^2\ln t} \text{ and }\frac{\ln \ln t}{t^2}.$$ 
By equation \eqref{sys-eq}, these terms should be in the expansion of $y(t)$ as well. Therefore, we need to have some form of expansions with combinations of functions $t$, $\ln t$, and $\ln\ln t$.
In fact, even more general result holds true as showed in the following theorem.

\begin{theorem}\label{mixpl}
 Let $(\mu_k)_{k=1}^\infty$ be a divergent,  strictly increasing  sequence of positive numbers that preserves the addition and the unit increment. Let $(n_k)_{k=1}^\infty$ be an increasing sequence in $\N$. Suppose
\beq\label{fme}
f(t)\sim \sum_{k=1}^\infty p_k(\LL_{n_k}(t)) t^{-\mu_k },\quad\text{ where } p_k\in \classP(n_k,\R^n).
\eeq 

Then  the solution $y(t)$ admits the asymptotic expansion 
\beq\label{ume}
y(t)\sim \sum_{k=1}^\infty q_k(\LL_{n_k}(t)) t^{-\mu_k },
\eeq 
where $q_k\in \classP(n_k,\R^n)$ is defined by
\beq\label{mixq}
  q_k=
  \begin{cases}
A^{-1}p_1,&\text{ for $k=1$,}\\
A^{-1}\Big(\begin{displaystyle}
\sum_{m\ge 2}\ \sum_{\mu_{j_{1}}+\mu_{j_{2}}+\ldots + \mu_{j_{m}}=\mu_k}
\end{displaystyle} 
\mathcal G_m(  q_{j_{1}}, q_{j_{2}},\ldots ,  q_{j_{m}} )   + p_k + \chi_k\Big),
  &\text{ for  $k\ge 2$,}
  \end{cases}
\eeq
with $\chi_k\in\classP(n_k,\R^n)$ being
\beq\label{chime}
\chi_k(z)=\begin{cases}
\begin{displaystyle}
 \mu_{\lambda} q_\lambda(z) - \sum_{j=1}^{n_\lambda} \frac{1}{z_1 z_2\ldots z_{j-1}}\cdot \frac{\partial q_\lambda(z)}{\partial z_j},
\end{displaystyle}
&\text{ if there is $\lambda\ge 1$}\\
&\text{such that $\mu_\lambda +1 = \mu_k,$} \\
 0,&\text{otherwise},
 \end{cases}
\eeq
for $z=(z_1,z_2,\ldots,z_{n_k})\in (0,\infty)^{n_k}$.
\end{theorem}
\begin{proof}
Similar to the verification after Theorem \ref{logthm}, and thanks to Lemma \ref{Pprop}\ref{P4},\ref{P6}, we can validate that $q_k$ and $\chi_k$ belong to $\classP(n_k,\R^n)$ for all $k\in\N$. 

Set $\psi(t)=t$.
For $k\in \N$, we denote 
$\tilde p_k(t)=p_k(\LL_{n_k}(t))$ and define $f_k(t)$ and $\bar f_N(t)$ as in \eqref{fkdef}.
Again, one can verify \eqref{ffirst} and, hence,  \eqref{u-first}.

It suffices to prove, for all $N\in\N$, that
\beq\label{TNmix}
\Big| y(t)- \sum_{k=1}^N q_k(\LL_{n_k}(t)) t^{-\mu_k }\Big|=\bigo(t^{-\mu_N-\delta_N})
\text{ for some $\delta_N>0$. }
\eeq

We prove it by induction. Let ($\mathcal T_N$) denote the statement \eqref{TNmix}. Again, we consider $t$ sufficiently large for the rest of the proof.

\textbf{First step: $N=1.$} Let $w_0(t)=t^{\mu_1}y(t)$. Same as \eqref{woeq} of  Theorem \ref{mainthm} for Type 2, we have
  \beq\label{wx1}
 w_0' =-Aw_0 +  p_1(\LL_{n_1}(t))  +  \bigo(t^{-\widehat\delta_1})\text{ for some }\widehat\delta_1>0.
 \eeq

Applying Lemma \ref{iterlog-odelem} to equation \eqref{wx1}, there exists $\delta_1 >0$ such that
 \beq\label{wx2}
 \left |w_0(t) - A^{-1}p_1(\LL_{n_1}(t))\right |= \bigo(t^{-\delta_1 }).
  \eeq
  
Dividing \eqref{wx2} by $t^{\mu_1}$ gives
\beqs
\left |y(t)- t^{-\mu_1 }q_1(\LL_{n_1}(t))\right |=\bigo(t^{-\mu_1 - \delta_1}).
\eeqs

This estimate proves ($\mathcal T_1$).

\textbf{ Induction step: $N\ge 1$.}  Assume ($\mathcal T_N$).
Denote
$\tilde q_k(t)=q_k(\LL_{n_k}(t))$ and define $y_k(t)$, $\bar y_N$, $v_N$ as in \eqref{ys}.

Let $w_N(t)=t^{\mu_{N+1}}v_N(t)$.
Same as \eqref{w100} for Type 2, 
\beq\label{wx}
\begin{aligned}
w_N'&= -Aw_N + \tilde{Q}_{N+1} +\tilde p_{N+1}+ t^{\mu_{N+1}} \sum_{k=1}^N  t^{-\mu_k} \Big\{-A\tilde q_k + \tilde{Q}_k+\tilde p_k \Big\}\\
&\quad -  t^{\mu_{N+1}}\sum_{k=1}^N y_k'
+\bigo(t^{-\varep''_{N+1}}),
\end{aligned}
\eeq
for some $\varep''_{N+1}>0$.  Similar to \eqref{QQ}, we have, for $1\le k\le N+1$,  the function $\tilde Q_k(t)$ is $Q_k(\LL_{n_k}(t))$ with $Q_k$ defined by \eqref{Ik2}.

We calculate
\beq\label{ypk}
y'_k(t)=  \Big[ -\mu_k q_k(\LL_{n_k}(t)) + \sum_{j=1}^{n_k} \frac{1}{\iln_1 (t) \iln_2 (t)\ldots \iln_{j-1} (t)}\cdot \frac{\partial q_k}{\partial z_j}(\LL_{n_k}(t)) \Big] t^{-\mu_k-1}.
\eeq

Let $\tilde \chi_k(t)= \chi_k(\LL_k(t))$ for $1\le k\le N+1$. 
Note in \eqref{ypk} that $\mu_k+1=\mu_s$ for some $s\in \N$.
Summing up \eqref{ypk} in $k$ from $1$ to $N$ and split the sum into three parts corresponding to $s\le N$, $s=N+1$ and $s\ge N+2$, we obtain
\beqs
\sum_{k=1}^N y_k'(t)
 = - \sum_{k=1}^N t^{-\mu_k }\tilde \chi_k(t)
 -  t^{-\mu_{N+1}} \tilde \chi_{N+1}(t)+\bigo(t^{-\mu}),
 \eeqs
for some $ \mu > \mu_{N+1}$. Thus
\beq\label{ymul}
t^{\mu_{N+1}}\sum_{k=1}^N y_k'(t)
=-t^{\mu_{N+1}}\sum_{k=1}^N t^{-\mu_k }\tilde \chi_k(t)
 - \tilde \chi_{N+1}(t) + \bigo(t^{-\mu+\mu_{N+1}}).
\eeq

From \eqref{ymul} and \eqref{wx}, it follows
\beqs
w_N'
=-Aw_N + \hat \Phi_{N+1}(\LL_{n_{N+1}}(t))  +t^{\mu_{N+1}} \sum_{k=1}^N t^{-\mu_k}\Phi_k(\LL_{n_k}(t))
+ \bigo(t^{-\widehat\delta_{N+1}}),
\eeqs
where $\widehat\delta_{N+1}=\min\{\varep''_{N+1},\mu-\mu_{N+1}\}>0$,
\beqs
\hat \Phi_{N+1}=Q_{N+1}+  p_{N+1} + \chi_{N+1},  
\eeqs
\beqs
\Phi_k=-Aq_k+ p_k+Q_k + \chi_k,\quad 1\le k\le N.
\eeqs

By \eqref{mixq}, we have $\Phi_k=0$ for $k=1,\ldots,N$. Thus,
\beqs
w_N'=-A w_N + \hat \Phi_{N+1}(\LL_{n_{N+1}}(t)) + \bigo(t^{-\widehat\delta_{N+1}}).
\eeqs

Applying Lemma \ref{iterlog-odelem} to this equation for $w_N$,  there exists $\delta_{N+1}>0$ such that
\beq \label{wM}
\left|w_N(t)-A^{-1}\hat \Phi_{N+1}(\LL_{n_{N+1}}(t))\right|=\bigo(t^{-\delta_{N+1}}).
\eeq

Noting that $A^{-1}\hat \Phi_{N+1}=q_{N+1}$, and dividing equation \eqref{wM} by $t^{\mu_{N+1}}$  give
\beqs
\left|v_N(t)-t^{-\mu_{N+1} } q_{N+1}(\LL_{n_{N+1}}(t))\right|=\bigo(t^{-\mu_{N+1} - \delta_{N+1}}).
\eeqs

This implies ($\mathcal T_{N+1}$) and completes the induction proof of \eqref{TNmix} for all $N\in\N$.
\end{proof}

\begin{example}\label{eg55}
 Case $n_k=2$ for all $k\in\N$. Assumption \eqref{fme} becomes
\beq\label{f2}
f(t)\sim \sum_{k=1}^\infty p_k(\ln t,\ln\ln t) t^{-\mu_k }\quad \text{with } p_k\in \classP(2,\R^n).
\eeq 
Then the conclusion \eqref{ume} becomes
\beq\label{u2}
y(t)\sim \sum_{k=1}^\infty q_k(\ln t,\ln\ln t) t^{-\mu_k },
\eeq 
where each $q_k\in\classP(2,\R^n)$ is defined by \eqref{mixq}, with $\chi_k$ in \eqref{chime} becoming
\beq\label{x2}
\chi_k(z)=\begin{cases}
\begin{displaystyle}
 \mu_{\lambda} q_\lambda(z) - \frac{\partial q_\lambda(z)}{\partial z_1} - \frac{1}{z_1}\cdot \frac{\partial q_\lambda(z)}{\partial z_2},
\end{displaystyle}
&\text{ if there exists $\lambda\ge 1$ such that }\\
&\quad \mu_\lambda +1 = \mu_k, \\
 0,&\text{otherwise},
 \end{cases}
\eeq
for $z=(z_1,z_2)\in (0,\infty)^2$.
\end{example}

\begin{corollary}\label{cor54}
 Let $(\mu_k)_{k=1}^\infty$ be  as in Theorem \ref{mixpl}.
Assume 
\beq \label{fb0}
f(t) \sim \sum_{k=1}^\infty \frac{1}{t^{\mu_k}}\Big( \sum_{j=1}^{n_k}  \frac{p_{k,j}(\ln \ln t)}{\ln^{\beta_{k,j}}(t)}\Big),
\eeq
where $n_k\ge 1$, $\beta_{k,j}\ge 0$, and $p_{k,j}$'s are polynomials from $\R$ to $\R^n$ for $1\le j\le n_k$.

Then there exist natural numbers $J_k$'s, for $k\in\N$, $\R^n$-valued polynomials of one variable $q_{k,j}$'s and non-negative numbers $\gamma_{k,j}$'s, for $1\le j\le J_k$, such that 
\beq \label{uqln}
y(t) \sim \sum_{k=1}^\infty \frac{1}{t^{\mu_k}}\Big( \sum_{j=1}^{J_k} \frac{q_{k,j}(\ln \ln t)}{\ln^{\gamma_{k,j}}(t)}\Big).
\eeq
 \end{corollary}
\begin{proof}
We follow Example \ref{eg55}. Expansion \eqref{fb0} is of the form \eqref{f2}, where  
\beqs
p_k(z_1,z_2)=\sum_{j=1}^{n_k} p_{k,j}(z_2)z_1^{-\beta_{k,j}},
\eeqs
with $p_{k,j}$ being $\R^n$-valued polynomials of one variable.
Then we have expansion \eqref{u2}. We prove that \eqref{u2}, in fact, is of the form \eqref{uqln}.
Define 
\beq \label{Fdef}
\begin{aligned}
F&=\Big\{h:(0,\infty)^2\to \R^n \text{ such that } h(z_1,z_2)\text{ is a finite sum of the products }q(z_2)z_1^{-\gamma}\\ 
&\qquad \text{ for some polynomial $q:\R\to\R^n$ and number $\gamma\ge 0$} \Big\}.
\end{aligned}
\eeq 

To establish \eqref{uqln}, it suffices to have that $q_k\in F$ for all $k\in\N$.
We will prove this by induction in $k$. First, we have $p_k\in F$ for all $k\in\N$.

Because $q_1=A^{-1}p_1$ and $p_1\in F$, we have $q_1\in F$. Let $k\ge 2$, and assume $q_s\in F$ for all $1\le s\le k-1$. 
 Then, for $1\le s\le k-1$, we can write 
 \beqs
 q_s(z_1,z_2)=\sum_{j=1}^{J_s} q_{s,j}(z_2) z_1^{-\gamma_{s,j}},
 \eeqs
where  $q_{s,j}$'s are $\R^n$-valued polynomials of one variable, and $\gamma_{s,j}\ge 0$. 
Note that
\beq\label{qsdev}\begin{aligned}
\frac{ \partial q_s(z_1,z_2)}{\partial z_1}&=\sum_{j=1}^{J_s} -\gamma_{s,j} q_{s,j}(z_2) z_1^{-\gamma_{s,j}-1}, \\
\frac1{z_1}\frac{ \partial q_s(z_1,z_2)}{\partial z_2}&=\sum_{j=1}^{J_s}  q_{s,j}'(z_2) z_1^{-\gamma_{s,j}-1}. 
 \end{aligned}
 \eeq 

 Consider $q_k$ defined by \eqref{mixq} and \eqref{x2}. 
It follows formula \eqref{x2} and calculations in \eqref{qsdev} that $\chi_k\in F$.  Then it is obvious that each term on the right-hand side of \eqref{mixq} belongs to $F$, therefore, so does $q_k$.
By the induction principle, we conclude $q_k\in F$ for all $k\in\N$ and complete the proof.
\end{proof}

Below is a particular case  when the polynomial $q_{k,j}$ in \eqref{uqln} can be determined more explicitly.

\begin{corollary}\label{cor55}
Let $(\mu_k)_{k=1}^\infty$ and  $(\beta_j)_{j=1}^\infty$ be divergent,  strictly increasing  sequences that preserve the addition and the unit increment, with $\mu_1>0$ and $\beta_1\ge 0$. Assume
\beq \label{fb1}
f(t) \sim \sum_{k=1}^\infty \frac{1}{t^{\mu_k}} \Big(\sum_{j=1}^\infty  \frac{p_{k,j}(\ln \ln t)}{\ln^{\beta_j}(t)}\Big),
\eeq
where $p_{k,j}$'s are the polynomials which, for each $k\in\N$,  differ from zero for only finitely many $j$'s.
Then
\beq \label{uMul}
y(t) \sim \sum_{k=1}^\infty \frac{1}{t^{\mu_k}}\Big( \sum_{j=1}^\infty \frac{ q_{k,j}(\ln \ln t)}{ \ln^{\beta_j}(t)}\Big),
\eeq
where $q_{k,j}$'s are the polynomials defined by
\beq \label{q1j}
q_{1,j}=A^{-1}p_{1,j}\quad \text{for } j\in\N,
\eeq
and, for $ k \ge 2$, $j\in\N$,
\beq  \label{qkj} 
\begin{aligned}
 q_{k,j}
&=A^{-1}\Big\{ \sum_{m\ge 2}\ \sum_{ \substack{ \mu_{j_{1}}+\mu_{j_{2}}+\ldots \mu_{j_{m}}=\mu_{k} \\
\beta_{l_1} + \beta_{l_2} + \ldots \beta_{l_m}=\beta_j}} \mathcal G_m(  q_{j_{1},l_1}, q_{j_{2},l_2},\ldots ,  q_{j_{m},l_m})   + p_{k,j} + \chi_{k,j}\Big\},
\end{aligned}
\eeq 
with
\beq\label{x3}
 \chi_{k,j}=
 \begin{cases}
\mu_{\lambda}q_{\lambda,j} - q'_{\lambda,\ell} + \beta_\ell q_{\lambda,\ell},&\text{ if there exist } \lambda,\ell\ge 1
\text{ such that }\\
&\quad \mu_\lambda +1 = \mu_k,\ 
\beta_\ell+1=\beta_j, \\
 0,&\text{otherwise}.
 \end{cases}
 \eeq
 \end{corollary}
\begin{proof}
First of all, one can verify, by induction in $k$, that the function $q_{k,j}$, for each $k\in\N$, differs from the zero function only for finitely many $j$'s.

We apply Corollary \ref{cor54} to $\beta_{k,j}=\beta_j$ for all $k,j$.
Under our assumptions, expansion \eqref{fb1} is the same as \eqref{f2} with
\beq\label{p4}
p_k(z_1,z_2)=\sum_{j=1}^\infty p_{k,j}(z_2)z_1^{-\beta_j}.
\eeq
Then we have expansion \eqref{u2} with $q_k$ defined by \eqref{mixq} and \eqref{x2}.
It suffices to establish, for all $k\in\N$, that
\beq\label{qsum}
q_k(z_1,z_2)=\sum_{j=1}^\infty q_{k,j}(z_2) z_1^{-\beta_j},
\eeq
where $q_{k,j}$ are defined by \eqref{q1j} and \eqref{qkj}. We prove \eqref{qsum} by induction.

When $k=1$, thanks to \eqref{mixq} and \eqref{p4},
\beqs
q_1(z_1,z_2)=A^{-1}p_1(z_1,z_2)=\sum_{j=1}^\infty A^{-1}p_{1,j}(z_2)z_1^{-\beta_j},
\eeqs
thus \eqref{qsum} holds true for $k=1$.

Let $k\ge 2$, assume formula \eqref{qsum} is true for $q_1,q_2, \ldots,q_{k-1}$.

Set $\mathcal B=\{ \beta_j:j\in\N\}$.
Let $F$ be the set of functions defined by \eqref{Fdef} with the restriction $\gamma\in \mathcal B$.
Because $\mathcal B$ preserves the addition, we have $\mathcal G_m(h_1,\ldots,h_m)\in F$ whenever $h_1,\ldots,h_m\in F$.
Therefore, the double sum in \eqref{mixq} belongs to $F$ and can be rewritten as
\beq\label{Gs4}
\sum_{j=1}^\infty \frac1{z_1^{\beta_j}} \Big\{\sum_{m\ge 2}\ \sum_{ \substack{ \mu_{j_{1}}+\mu_{j_{2}}+\ldots+ \mu_{j_{m}}=\mu_{k}\\
\beta_{l_1} + \beta_{l_2} + \ldots +\beta_{l_m}=\beta_j}} \mathcal G_m(  q_{j_{1},l_1}, q_{j_{2},l_2},\ldots ,  q_{j_{m},l_m})\Big\}.
\eeq 

We have from \eqref{x2} that
\beq\label{x4}
 \chi_k(z_1,z_2)=\mu_\lambda \sum_{j=1}^\infty \frac{q_{\lambda,j}(z_2)}{z_1^{\beta_j}}
 +\sum_{j=1}^\infty \beta_j \frac{q_{\lambda,j}(z_2)}{z_1^{\beta_j+1}}  
 -\frac1{z_1}\sum_{j=1}^\infty \frac{q_{\lambda,j}'(z_2)}{z_1^{\beta_j}},
\eeq
if there is $\lambda$ such that $\mu_\lambda+1=\mu_k$,
or $\chi_k(z_1,z_2)=0$, otherwise.

Since $\beta_j+1\in \mathcal B$ for all $j$, we can rewrite \eqref{x4} as
\beq\label{x5}
\chi_k(z_1,z_2)=\sum_{j=1}^\infty \frac{\chi_{k,j}(z_2)}{z_1^{\beta_j}},
\eeq
where $\chi_{k,j}$, for  $j\in\N$, are polynomials defined as in formula \eqref{x3}.

Combining \eqref{mixq} with \eqref{Gs4}, \eqref{p4} and \eqref{x5}, we obtain formula \eqref{qsum} for $q_k$.

By the induction principle, \eqref{qsum} holds true for all $k\in\N$. The proof is complete.
\end{proof}

\begin{example}\label{lasteg}
 Suppose $\mu_k=k$ and $\beta_j=j$ for all $k,j\in\N$.
 Then $q_{1,j}=A^{-1}p_{1,j}$ for all $j\in\N$,
and, for $ k \ge 2$, $j\in\N$,
\beqs  
 q_{k,j}=A^{-1}\Big\{ \sum_{m\ge 2}\ \sum_{ \substack{j_{1}+j_{2}+\ldots + j_{m}={k} \\
{l_1} + {l_2} + \ldots + {l_m}=j}} \mathcal G_m(  q_{j_{1},l_1}, q_{j_{2},l_2},\ldots ,  q_{j_{m},l_m})   + p_{k,j} + \chi_{k,j}\Big\},
\eeqs 
where
$ \chi_{k,j}=(k-1)q_{k-1,j} - q'_{k-1,j-1} + (j-1)q_{k-1,j-1}$.
\end{example}

\appendix

\section{Appendix}\label{append}

We discuss a particular application of our results to numerical approximations of a nonlinear PDE problem using ODE systems. 
The presentation below is focused on the ideas without showing technical details.

Consider the Navier--Stokes equations \eqref{Nf} with a given initial data $u(0)=u_0$.  For $m\in\N$, let $P_m$ denote the orthogonal projection to the first $m$ eigenspaces (corresponding to the first $m$ distinct eigenvalues) of the Stokes operator $A$.

The Galerkin approximation problem is
\beq\label{galerkin}
\frac{d u_m}{dt} + Au_m +B_m(u_m,u_m)=P_m f,\quad u_m(0)=P_m u_0, 
\eeq
where $B_m(u,u)=P_mB(u,u)$. For each $m\in\N$, the approximate system \eqref{galerkin} is an ODE system in a finite dimensional space, and $B_m(\cdot,\cdot)$ is a bilinear form. Thus, the results obtained in previous sections apply.
 
Consider Type 1, 2, 3 problems as in section \ref{basetwo}, that is,
\beq
f(t)\sim \sum_{k=1}^\infty p_k(\phi(t))\psi(t)^{-k},
\eeq
where  the base functions $\phi(t)$ and $\psi(t)$ are given in Definition \ref{typedef}.

Then the solutions $u(t)$ and $u_m(t)$ have the asymptotic expansions
\beq
u(t)\sim \sum_{k=1}^\infty q_k(\phi(t))\psi(t)^{-k}\text{ and }
u_m(t)\sim \sum_{k=1}^\infty q_k^{(m)}(\phi(t))\psi(t)^{-k}, \text{ respectively.}
\eeq

The question is whether $q_k^{(m)}$ converges to $q_k$ as $m\to\infty$ in certain sense.

First,  we roughly have 
\beq\label{Bm} 
B_m(u,u)\to B(u,u),\ P_m u_0\to u_0\text{ and }P_mf\to f \text{ as $m\to\infty$.}
\eeq 
(The normed spaces in which the convergences hold depend on the regularity of $u$, $u_0$ and $f$.)

For Types 2 and 3, the polynomials $q_k$'s are independent of the solution $u(t)$, depend only on $p_k$ and $B(\cdot,\cdot)$.
Similarly, for each $m\in\N$, the polynomials $q_k^{(m)}$'s are independent of the individual solution $u_m(t)$, depend only on $P_mp_k$ and $P_mB(\cdot,\cdot)$. With the convergences in \eqref{Bm} and explicit formulas \eqref{qtype2} and \eqref{qtype3}, it is likely that the coefficients of $q_k^{(m)}$ converge to its corresponding coefficients of $q_k(t)$, as $m\to\infty$.

For Type 1, we consider the case $u(t)$ is a unique, regular solution on $[0,\infty)$.
The construction of polynomial $q_k$, respectively $q_k^{(m)}$, depends on the long-time values of $u(t)$, respectively $u_m(t)$. 
Therefore, determining the convergence of $q_k^{(m)}$ to $q_k$, as $m\to\infty$, is more subtle than in the case of Types 2 and 3.
However, we only consider the convergence for each fixed $k$, and, in light of many related estimates in previous work such as \cite{FHOZ1,FHOZ2,FS91,FS84b}, it may still be possible to prove such a convergence.  


\def\cprime{$'$}\def\cprime{$'$} \def\cprime{$'$}

\end{document}